\documentclass[final,3p,times]{elsarticle}
\usepackage{lineno,hyperref}
\usepackage{amsfonts}
\usepackage{enumerate}
\usepackage{subfigure}
\usepackage{amsmath}
\usepackage{mathtools}
\allowdisplaybreaks[4]
\usepackage{amsthm}
\usepackage{mathrsfs}
\usepackage{bm}
\usepackage{float}
\usepackage{dsfont}
\usepackage{cite}
\usepackage{booktabs}
\usepackage{bbm}
\usepackage{multicol}
\usepackage{multirow}
\usepackage{tabu}
\usepackage{cases}
\usepackage{diagbox}
\usepackage{url}
\usepackage{graphics}
\usepackage{color}

\biboptions{numbers,sort&compress}

\journal{EJOR}
\biboptions{authoryear}

\newtheorem{defi}{Definition}

\newtheorem{assum}{Assumption}

\newtheorem{lem}{Lemma}
\newtheorem{theo}{Theorem}
\newtheorem{coro}{Corollary}

\DeclareMathOperator*{\minimize}{minimize}
\DeclareMathOperator*{\maximize}{maximize}
\DeclareMathOperator*{\st}{subject\ to}

\begin{document}

\begin{frontmatter}

\title{Wasserstein Distributionally Robust Shortest Path Problem}

\author[thu]{Zhuolin Wang}
\ead{wangzl17@mails.tsinghua.edu.cn}
\author[thu]{Keyou You\corref{cor1}}
\ead{youky@tsinghua.edu.cn}
\author[thu]{Shiji Song}
\ead{shijis@tsinghua.edu.cn}
\author[bit]{Yuli Zhang}
\ead{zhangyuli@bit.edu.cn}

\cortext[cor1]{Corresponding author.}
\address[thu]{Department of Automation and BNRist, Tsinghua University, Beijing 100084, P.R. China}
\address[bit]{School of Management and Economics, Beijing Institute of Technology, Beijing 100081, P.R. China}

\begin{abstract}
This paper proposes a data-driven distributionally robust shortest path (DRSP) model where the distribution of the travel time in the transportation network can only be partially observed  through a {\em finite}  number of samples. Specifically, we aim to find an optimal path to minimize the worst-case $\alpha$-reliable mean-excess travel time (METT) over a Wasserstein ball, which is centered at the empirical distribution of the sample dataset and the ball radius quantifies the level of its confidence. {In sharp contrast to the existing DRSP models, our model is equivalently reformulated as a tractable mixed 0-1 convex problem, e.g., 0-1 linear program or 0-1 second-order cone program. Moreover, we also explicitly derive the distribution achieving the worst-case METT by simply perturbing each sample. Experiments demonstrate the advantages of our DRSP model in terms of the out-of-sample performance and computational complexity. Finally, our DRSP model is easily extended to solve the DR bi-criteria shortest path problem and the minimum cost flow problem.}

\end{abstract}

\begin{keyword}
Distributionally robust shortest path \sep ambiguity set \sep Wasserstein metric \sep METT \sep mixed 0-1 convex program  
\end{keyword}

\end{frontmatter}


\section{Introduction}
The shortest path problem is one of the most fundamental problems in the transportation network and has broad applications, see e.g. \citet{baxter2014incremental,tilk2017asymmetry,Cao2016Finding}. To obtain an optimal path,  the travel time in each arc of the network is essential. Due to different weather conditions, path capacity, traffic control and etc,  travel time is usually subject to large variabilities, which may greatly affect the  selection of an optimal path. In fact, travelers are not only concerned with the ``nominal" travel time of each path but also its reliability \citep{bertsekas1991analysis,Fosgerau2011The,nikolova2014mean}.

In the literature, stochastic shortest path (SSP) models have been proposed to handle the random uncertainty in travel time under different criteria such as effective travel time \citep{lo2003network}, percentile travel time \citep{frank1969shortest} and mean-excess travel time \citep{chen2010alpha}. These models require an exact distribution of travel time for finding an optimal path. In practice, it is difficult to obtain the exact distribution since it may be time-varying and can only be estimated through a finite sample dataset \citep{masson2006inferring}. Thus, a natural method is to approximate the SSP model via the sample-average approximation (SAA) method \citep{shapiro1998simulation,verweij2003sample}, where the true distribution is approximated by the discrete empirical distribution over the sample dataset. This method is only applicable to situations where the distribution is time-invariant and a large number of good samples can be generated cheaply.  When the sample dataset is of low quality, the empirical distribution may significantly deviate from the true distribution and the SAA method tends to exhibit poor out-of-sample performance. From this perspective, the SAA method is not always reliable.

An alternative approach is to apply the distributionally robust (DR) optimization technique to the shortest path problem \citep{cheng2013distributionally,shahabi2015robust,yang2017optimizing,Zhang2017Robust}, leading to a distributionally robust shortest path (DRSP) model. The DRSP model assumes that the true distribution belongs to an ambiguity set of distributions, over which an optimal path is to be found in some worst-case sense, e.g., the one minimizes the worst-case $\alpha$-reliable mean-excess travel time (METT)  \citep{Zhang2017Robust}.  For instance, the true distribution in \citet{chassein2019algorithms} is parameterized with a vector that is assumed to lie in some set constructed by exploiting samples.  However, it is NP-hard  to solve most of the DRSP models \citep{yu1998robust} and only a few DRSP models are tractable by some well-defined ambiguity set \citep{esfahani2018data}.  

In general, the ambiguity set should be large enough to include the true distribution with a high probability but can not be too ``large" to avoid too conservative decisions. Moreover, the associated DRSP model needs to be as tractable as possible. The moment-based ambiguity set which consists of distributions with specified moment constraints is adopted for the DR optimization problem \citep{delage2010distributionally}.  Specifically, the DRSP model in \citet{cheng2013distributionally,Zhang2017Robust} assumes that the ambiguity set contains distributions with the {\em exactly} known first and second moments. Observe that this may lead to poor decisions if mismatch moments are used for the ambiguity set. In fact, it is often the case that we cannot obtain exact moment information with a finite sample dataset. Thus, one inevitably needs to  further introduce uncertainty in the moment, which easily renders their DRSP models intractable.  To resolve it, a metric-based ambiguity set has been developed in  \citet{Hu2012Kullback,jiang2016data,erdougan2006ambiguous,pflug2007ambiguity,wozabal2012framework}.

%
In this paper, we propose a novel data-driven DRSP model with a metric-based ambiguity set for the unknown distribution of travel time, which is defined as a ball centered at the empirical distribution over the finite sample dataset and the ball radius reflects our confidence in the empirical distribution. That is, the higher the quality of the empirical distribution, the smaller the ball radius. This facilitates us to utilize the sample dataset in a flexible way to hedge uncertainty. Then, we further incorporate the support set information into our DRSP model where the travel time is restricted to an interval constructed from the sample dataset. Noting that the empirical distribution is discrete and the true distribution of travel time is usually continuous, we adopt the Wasserstein metric \citep{kantorovich1958space} to measure the distance between {\em any} two distributions, which is different from the Kullaback-Leibler divergence \citep{Hu2012Kullback,jiang2016data} and the Prokhorov metric  \citep{erdougan2006ambiguous}. 

Then, our DRSP model aims to minimize the worst-case METT  over the aforementioned Wasserstein ball. While the moment-based DRSP model in \citet{Zhang2017Robust} is  NP-hard, our  model can be equivalently reformulated as a solvable and finite mixed 0-1 convex program, e.g., 0-1 linear program (LP) or 0-1 second-order cone program (SOCP). Different from \citet{Zhang2017Robust}, the distribution that achieves the worst-case METT can also be explicitly derived by simply perturbing each sample, which is very helpful to assess the quality of the optimal path.  Since the Wasserstein ball includes the  true distribution with a high probability, 
the optimal path of our model offers good out-of-sample performance. Moreover, it asymptotically converges to the solution of the SSP model under the true distribution if the size of the sample dataset tends to infinity.  Finally, experiment results show that our model achieves a better out-of-sample performance than that of the moment-based DRSP model in \citet{Zhang2017Robust} and the SAA method.  Our method is easily extended to solve the robust bi-criteria shortest path problem and the robust minimum cost flow problem over the Wasserstein ball. Both problems are reduced to finite convex programs which can be solved efficiently by existing algorithms. A preliminary version of this work has been presented in \citet{Wang2019datadriven}, which only introduces the DRSP model over the Wasserstein set. Overall, the main contributions of this paper are summarized below.
\begin{enumerate}[(1)]
	\item We propose a novel data-driven DRSP model which aims to find an optimal path to minimize the worst-case METT within a Wasserstein ball. 
	\item In contrast to the existing DRSP models, our model is equivalently reformulated as a solvable mixed 0-1 convex problem and the worst-case distribution attaining the worst-case METT can be explicitly derived. 
	\item Besides its good out-of-sample performance and low computational complexity,  our model can be easily extended to solve the robust bi-criteria shortest path problem as well as minimum flow cost problems.	 
\end{enumerate}

The rest of the paper is organized as follows. Section \ref{PF} proposes the DRSP model to minimize the worst-case METT over a Wasserstein ball. In Section \ref{Mod-out-sup}, we reformulate our DRSP model as a mixed 0-1 convex problem. The distribution achieving the worst-case METT is explicitly derived  in Section \ref{worst-dis}. In Section \ref{EtOP}, our model is extended to solve other DR problems. In Section \ref{exper}, we perform numerical experiments to illustrate the performance of our model. Some conclusion remarks are drawn in Section \ref{con}.

\section{Problem Formulation} \label{PF}

\subsection{The Shortest Path Problem and Reliability Criteria}\label{class-shortest-path}
Consider a directed and connected network $\mathcal{G = (V, A)}$ with the vertex set $\mathcal{V}$ and the arcs set $\mathcal{A}$, where  $|\mathcal{V}| = m $ and $|\mathcal{A}| = n$. Let ${\xi}_{ij}$ denote the travel time on an arc $(i,j)$ in $\mathcal{A}$ and ${\bm{\xi}}=\{\xi_{ij}: (i,j)\in \mathcal{A} \}$ be the stacked vector of the travel time over all arcs. A directed path is  a sequence of arcs which connect a sequence of vertices in the same direction. Let $\bm{p} = \{p_{ij}:(i,j)\in \mathcal{A}\}$ represent a directed path from the origin vertex $o$ to the destination vertex $d$, where $p_{ij}$ is a binary decision variable and is one if arc $(i,j)$ is on the path from vertex $o$ to vertex $d$.

The standard shortest path problem aims to find an optimal path such that the sum of the travel time along the path is minimized \citep{ahuja1995network}, i.e.,

	\begin{align}\label{ssp1}
		\minimize_{\bf{p}} \ \ \  &\sum\limits_{(i,j)\in \mathcal{A}}{\xi}_{ij}p_{ij} \\
		 \st \ \ \ & \sum\limits_{j:(i,j)\in \mathcal{A}}p_{ij} - \sum\limits_{j:(j,i)\in \mathcal{A}}p_{ji} = b_i,  \forall i\in \mathcal{V}, \label{1b}\\
		& p_{ij} \in \{0,1\}, \forall (i,j)\in \mathcal{A} \label{1d},
	\end{align}
where $b_o=1, b_d = -1$ and $b_i = 0$ for $i \in \mathcal{V}/\{o,d\}$.
The constraint in (\ref{1b}) ensures the flow balance for the origin-destination pair $(o,d)$. In the sequel, let $\mathcal{P}$ be the set of feasible paths from the original vertex $o$ to the destination vertex $d$, i.e., 
\begin{equation}
\mathcal{P}=\{\bm{p}~|~\bm{p}~\text{satisfies}~\eqref{1b}~\text{and}~\eqref{1d}\}.
\end{equation}


The shortest path problem considers the network with a known vector of travel time, i.e., the constant vector ${\bm{\xi}}$ is assumed to be exactly given. In practice, travel time variability is unavoidable due to uncertain factors, e.g. traffic jams and weather conditions. Obviously, this vector has a  significant impact on finding an optimal path for travelers. For example, consider a simple network with only three vertices and three arcs in Figure \ref{fig:label}. The traveler aims to find an optimal path $\bm{p}^*$ from $1$ to $3$ with the random travel time ${\xi}_{ij}$.
\begin{figure}[h]
	\centering
	\includegraphics[scale=0.5]{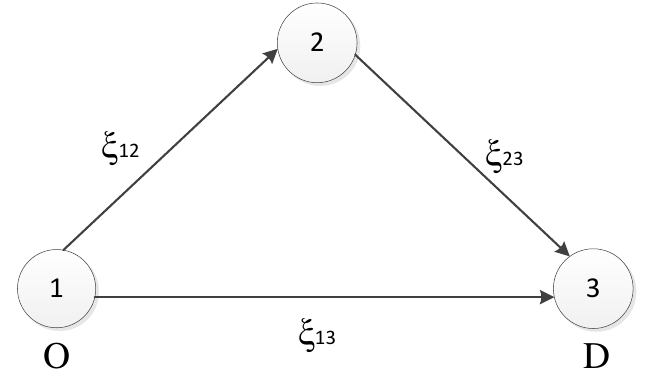}
	\caption{Topology of a simple network where \textsl{O} represents the origin vertex and \textsl{D} represents the destination vertex.}
	\label{fig:label}
\end{figure}
Suppose that a traveler observes ${{\xi}}_{13}=3.5$, ${{\xi}}_{12}=1.5$, ${{\xi}}_{23}=1.5$. Then an optimal path is obtained via solving \eqref{ssp1}, i.e.,  $1\rightarrow 2 \rightarrow 3$  where $a \rightarrow b$ represents the directed arc from vertex $a$ to vertex $b$. However, the travel time vector may change to ${{\xi}}_{13}=2.5,\ {{\xi}}_{12}=2,\ {{\xi}}_{23}=1.2$ and then the optimal choice is reset as $1 \rightarrow 3$. Thus, the optimal path in the above is not always reliable if the uncertainty in travel time is neglected.

To quantify the reliability of a path, some criteria have been proposed, such as effective travel time (ETT) \citep{lo2003network}, percentile travel time (PTT) \citep{lo2003network} and mean-excess travel time (METT) \citep{chen2010alpha}.  Under the assumption that ${\bm{\xi}}$ is a random vector with a distribution function $F$, the $\alpha$-reliable METT of a path and the corresponding SSP model are defined below.


\begin{defi}
	The $\alpha$-reliable METT of path $\bm{p}$ is defined as
	\begin{equation}
		\label{eq2}
		{\rm{METT}}_\alpha(\bm{p}) = \min_{t\in \mathbb{R}} \left\{t + \frac{1}{\alpha}\mathbb{E}_F  \left\{h(\bm{p},t,\bm{\xi})\right\}\right\},
	\end{equation}
	where $h(\bm{p},t,\bm{\xi}) = [{\bm{\xi}}^T\bm{p}-t]^+$ and $[x]^+ = \max\{x,0\}$. The associated SSP model is given by 
		\begin{equation}
		\label{SSP}
		\minimize_{\bm{p}\in \mathcal{P}} \ \ {\rm{METT}}_\alpha(\bm{p}).\\
	\end{equation}
\end{defi}
The METT which coincides with the conditional Value-at-Risk \citep{rockafellar2000optimization} is one of the most important criteria  to evaluate the path reliability. It is able to simultaneously address questions ``how much time do I need to allow" and ``how bad should I expect from the worst cases?" \citep{chen2010alpha}. However, solving the SSP model in \eqref{SSP} requires the exact distribution function $F$.

\subsection{Data-driven Robust Shortest Path Problem}

Usually, the true distribution $F$ in \eqref{SSP} is unavailable and can only be partially observed through a finite sample dataset $\{\hat{\bm{\xi}}^i\}_{i\in[N]}$ where $\hat{\bm{\xi}}^i$ is an independent sample of the random vector of travel time and $[N]=\{1,\ldots,N\}$.  In this case, a natural idea is to adopt the SAA method. Specifically,
 $F$ is approximated by an empirical distribution $F_N$ over the sample dataset, i.e.,
$$  {F}_{N}({\bm{\xi}})={\frac {1}{N}}\sum _{i=1}^{N}\bm {1} _{\{\hat{\bm{\xi}}^k\leq {\bm{\xi}}\}}
$$
where $
{\bm  {1}}_{{A}}$ is the indicator of event $A$.  Then the SSP model  in \eqref{SSP} is  approximately solved by
\begin{equation}
	\label{SAA}
	\minimize_{t\in \mathbb{R},~\bm{p} \in \mathcal{P}} \left\{ \ t+ \frac{1}{\alpha} \frac{1}{N}\sum\limits_{i=1}^N h(\bm{p},t,\hat{\bm{\xi}}^i)\right\}.
\end{equation}

By Glivenko-Cantelli theorem \citep{glivenko1933sulla,cantelli1933sulla}, the empirical distribution $F_N$ converges weakly to the true distribution $F$ as $N$ tends to infinity. This implies that the solution to the SAA model in (\ref{SAA}) converges to that of the SSP model in (\ref{SSP}).  That is, the SAA method is sensible only for the case where the true distribution $F$ can be well approximated by the empirical distribution.

When the size of the sample dataset is small and/or the sample $\hat{\bm{\xi}}^k$ is of low quality, the empirical distribution $F_N$ may deviate far from the true distribution $F$. More importantly, the distribution $F$ may not be constant and is {\em time-varying}. Then, an optimal path of the SAA model in \eqref{SAA} may exhibit poor out-of-sample performance and is not always reliable.

As in \citet{esfahani2018data} for the continuous optimization problem, we adopt a data-driven robust approach to hedge against the path unreliability either from the uncertainty of travel time or its distribution. The key idea is that the true distribution $F$ is expected to ``close" to the empirical distribution $F_N$ with a high probability. In particular, we believe that $F$ may belong to an ambiguity set $\mathcal{F}_N$ that is centered at the empirical distribution $F_N$ and its size reflects our confidence in $F_N$. The higher the confidence of $F_N$, the smaller the ambiguity set $\mathcal{F}_N$. 

Since the true distribution $F$ is usually continuous and the empirical distribution $F_N$ is discrete, we adopt the Wasserstein metric \citep{kantorovich1958space} to evaluate their distance, leading to a Wasserstein ball  $\mathcal{F}_N$. Then, we are interested in the  worst-case METT (w-METT)  over the Wasserstein ball $\mathcal{F}_N$, i.e.,
\begin{equation}
\label{RMETT}
\begin{aligned}
\text{w-METT}_{\alpha}(\bm{p})& =\sup_{F\in \mathcal{F}_N} {\rm{METT}}_\alpha(\bm{p})\\
&= \min_{t\in\mathbb{R}}\left\{t + \frac{1}{\alpha}\sup_{F\in \mathcal{F}_N}\mathbb{E}_F  \left\{h(\bm{p},t,\bm{\xi})\right\}\right\},
\end{aligned}
\end{equation}
where the second equality follows from \citet{1806.07418}.  Then our DRSP model is obtained as
\begin{equation}
\label{eq3}
\begin{aligned}
\minimize_{\bm{p} \in \mathcal{P}}\ \ \text{w-METT}_{\alpha}(\bm{p}).
\end{aligned}
\end{equation}

To evaluate our optimal path, we aim to find the worst-case distribution $F^*$ that achieves the w-METT, i.e., 
$$ \sup_{F\in \mathcal{F}_N}\mathbb{E}_F  \left\{h(\bm{p},t,\bm{\xi})\right\} = \mathbb{E}_{F^*}  \left\{h(\bm{p},t,\bm{\xi})\right\}.$$

\subsection{Ambiguity Set via Wasserstein Metric}
The key of the DRSP model \eqref{eq3} is how to construct the Wasserstein ball
\begin{equation}
	\label{Wset}
	\mathcal{F}_N=\{F\in \mathcal{M}(\Xi):d_W(F_N,F)\le \epsilon_N\},
\end{equation}
where $\epsilon_N \ge 0$ reflects our confidence in the empirical distribution $F_N$ and $\mathcal{M}(\Xi)$ is the set of some probability distributions supported on $\Xi$, i.e.,
$\mathcal{F}_N$ contains all distributions within the $\epsilon_N$-Wasserstein distance from $F_N$, in which the metric $d_W$ is defined as follows.

\begin{defi}
	The Wasserstein metric $d_W$:$\mathcal{M}(\Xi) \times \mathcal{M}(\Xi) \rightarrow \mathbb{R}_+$ is defined as
	\begin{equation*}
		\begin{aligned}
			 d_W(F_1,F_2) & = \inf\left\{\int_{\Xi\times\Xi} d({\bm{\xi}}_1,{\bm{\xi}}_2) K(\mathrm{d}{{\bm{\xi}}}_1,\mathrm{d}{{\bm{\xi}}}_2):\right.\\ 
			&\left.\int_{\Xi} K({{\bm{\xi}}}_1,\mathrm{d}{{\bm{\xi}}}_2)=F_1({\bm{\xi}}_1),
			\int_{\Xi} K(\mathrm{d}{{\bm{\xi}}}_1,{{\bm{\xi}}}_2)=F_2({\bm{\xi}}_2)\right \},
		\end{aligned}
	\end{equation*}
	where $(\Xi,d)$ is a Polish metric space, $K: \Xi\times\Xi \rightarrow  \mathbb{R}_+$ is the joint distribution of $F_1 \in \mathcal{M}(\Xi) $ and $F_2 \in \mathcal{M}(\Xi)$. Moreover, $d({\bm{\xi}}_1,{\bm{\xi}}_2)= {\| {\bm{\xi}}_1-{\bm{\xi}}_2 \|}_p $ where $\| \cdot \|$ represents $l_p$-norm on $\mathbb{R}^n$.
\end{defi}

Though $d_W$ satisfies the axioms of a metric, it may take the infinity value and thus is not a real distance. We need the following assumption on the set $\mathcal{M}$ \citep{ambrosio2013user}, under which $d_W$ is actually a distance metric.

\begin{assum} \label{assum1}
	For any distribution $ F \in \mathcal{M}(\Xi)$, it holds that  $$\int_{\Xi}{\| {\bm{\xi}}\|}_p F(\mathrm{d}{\bm{\xi}})<\infty.$$
\end{assum}

Assumption \ref{assum1} requires the first moment of the distribution $F$ to be finite, and only sacrifices little modeling power. 

To the best of our knowledge, we are the first to adopt the Wasserstein ball in the DRSP model.  We note that a similar line of this research for a traveling salesman problem is studied in \citet{carlsson2018wasserstein}.  Alternative metric-based ambiguity sets have also been adopted in DR optimizations,  e.g., Kullback-Leibler set in \citet{Hu2012Kullback,jiang2016data}, and Prokhorov set in \citet{erdougan2006ambiguous}. However, the Kullback-Leibler metric is unable to effectively evaluate the distance between a continuous distribution and a discrete one. In particular, it enforces the associated $\mathcal{F}_N$ to be a set of discrete distributions. The Prokhorov metric easily leads to an intractable shortest path model.

When the sample dataset is reasonably large, the Wasserstein ball $\mathcal{F}_N$ includes the true distribution with a high probability \citep{esfahani2018data} and thus the DRSP model is expected to exhibit good out-of-sample performance. Importantly, the proposed DRSP model can be equivalently reformulated as a mixed 0-1 program which is solvable via existing optimization techniques, e.g. the outer approximation decomposition algorithm \citep{duran1986outer}, the branch-and-bound method \citep{gupta1985branch} and the extended cutting plane method \citep{westerlund1995extended}.

\section{Reformulation of the DRSP Model} \label{Mod-out-sup}

In this section, we transform the DRSP model (\ref{eq3}) over the  Wasserstein ball $\mathcal{F}_N$ in (\ref{Wset}) with or without the support set to a finite mixed 0-1 convex problem respectively.

\renewcommand{\arraystretch}{1.2}  
\begin{table*}[htb]
	\centering
	\begin{tabular}{|c| p{3cm}<{\centering}| p{3cm}<{\centering}| p{3cm}<{\centering}| p{3cm}<{\centering}|}
		\hline
		\multirow{2}[0]{*}{Norm}&
		\multicolumn{2}{c}{Without Support Set }&\multicolumn{2}{|c|}{With Support Set }\cr\cline{2-5}
		&w-METT&DRSP Model&w-METT&DRSP Model\cr
		\hline
		\multirow{1}{1cm}{\centering $p=1$}&LP&Mixed 0-1 LP&LP&Mixed 0-1 LP\cr\hline
		\multirow{1}{1cm}{\centering $p=2$}&LP&Mixed 0-1 SOCP&SOCP&Mixed 0-1 SOCP\cr\hline
		\multirow{1}{1cm}{\centering $p=\infty$}&LP&{Mixed 0-1 LP}&LP&Mixed 0-1 LP\cr\hline
		\multirow{2}[0]{*}{Otherwise}&\multirow{2}[0]{*}{LP}&{Mixed 0-1 Convex Program} &{Mixed 0-1 Convex Program} &{ Mixed 0-1 Convex Program}\cr
		\hline
	\end{tabular}
	\caption{Equivalent problems of the our DRSP model, where $p$ represents the $l_p$-norm on $\mathbb{R}^n$.}
	\label{form}
\end{table*}

\subsection{The DRSP Model Without Support Set} \label{nosupport}

We first derive equivalent formulations for the w-METT (\ref{RMETT}) and the proposed DRSP model (\ref{eq3}) over the Wasserstein ball $\mathcal{F}_N$ without a support set.

\begin{theo}
	\label{thoe1}
	Under Assumption \ref{assum1}, the w-METT in (\ref{RMETT}) over the Wasserstein ball $\mathcal{F}_N$ can be computed by a finite linear programming (LP) problem

		\begin{subequations}
		\label{eqMETT}
		\begin{align}
			 \minimize_{t,\bm{s},\lambda} \ \ \ & t + \frac{1}{\alpha}\left\{\frac{1}{N}\sum\limits_{i=1}^{N} s_i + \lambda\epsilon_N \right\}  \\
			\st \ \ \  &\bm{p}^T\hat{\bm{\xi}}^i-t \le s_i , s_i\ge 0,\ \  \forall i \in [N]  \label{ub} \\
			& {\| \bm{p} \|}_q  \le \lambda, \label{uc}
		\end{align}
		\end{subequations}
	 where ${\| \cdot \|}_q$ is the dual of $l_p$-norm, i.e.,$1/p+1/q=1$.
	
	 Moreover, the DRSP model (\ref{eq3}) is equivalently reformulated as the following mixed 0-1 convex problem 
	\begin{equation}
		\label{UDP}
		\begin{aligned}
			 \minimize_{\bm{p},t,\bm{s},\lambda} \ \ \ & t + \frac{1}{\alpha}\left\{\frac{1}{N}\sum\limits_{i=1}^{N} s_i + \lambda\epsilon_N \right\}  \\
			 \st \ \ \ &\bm{p}\in \mathcal{P},  (\ref{ub})\ \  \text{and} \ \ (\ref{uc}).\\	
		\end{aligned}
	\end{equation}
\end{theo}

To prove Theorem \ref{thoe1},  a lemma is introduced below.
\begin{lem}
	\label{lem 1}
	For any $\bm{w}\in \mathbb{R}^n$, it holds that
	\begin{equation}
	\label{lem}
	\sup_{\bm{x}\in \mathbb{R}^n} \left\{\bm{w}^T\bm{x}-\lambda {\| \bm{x} \|}_p \right\} = \sup_{\bm{x}\in \mathbb{R}^n}\left\{ ({\|\bm{w} \|}_q-\lambda)  {\| \bm{x} \|}_p  \right\}.
	\end{equation}
\end{lem}

\begin{proof}
	By Lemma 1 in \citet{zhang2017lagrangian}, the maximum of $\sup_{{\| \bm{x} \|}_p =t}\left\{\bm{w}^T\bm{x}\right\}$ is explicitly given by $t {\| \bm{w} \|}_q $. Then we obtain 
	\begin{equation*}
		\begin{aligned}
			&\sup_{{\bm{x}}\in \mathbb{R}^n}\left\{\bm{w}^T{\bm{x}}-\lambda {\| \bm{x} \|}_p \right\}=\sup_{t\ge 0}\sup_{ \|\bm{x}\|_p =t}\left\{\bm{w}^T\bm{x}-\lambda {\| \bm{x} \|}_p \right\} \\
			&=\sup_{t\ge0}\sup_{ {\| \bm{x} \|}_p =t}\left\{\bm{w}^T\bm{x}-\lambda t\right\} = \sup_{t\ge0}\left\{ {t\| \bm{w} \|}_q -\lambda t\right\}\\
			 &= \sup_{ {\| \bm{x} \|}_p \ge 0}\left\{ ({\|\bm{w} \|}_q-\lambda)  {\| \bm{x} \|}_p  \right\} \\
			 &= \sup_{\bm{x} \in \mathbb{R}^n}\left\{ ({\|\bm{w} \|}_q-\lambda)  {\| \bm{x} \|}_p  \right\},
		\end{aligned}
	\end{equation*}
	which implies (\ref{lem}).
\end{proof}

Now, we are ready to prove Theorem \ref{thoe1}.

\begin{proof}[Proof of Theorem \ref{thoe1}]

	 For any feasible path $\bm{p} \in \mathcal{P}$, it is easily seen from \eqref{Wset} that $\sup_{F\in \mathcal{F}_N}\mathbb{E}_F  \left\{h(\bm{p},t,\bm{\xi})\right\}$ can be obtained by solving the following conic linear program
	\begin{equation}
		\label{eq8}
		\begin{aligned}
			 \maximize_{K({{\bm{\xi}}},~ {\hat{\bm{\xi}}}^i)\ge 0} \ \ \ & \int_{\Xi} \sum\limits_{i=1}^{N}h(\bm{p},t,\bm{\xi})K(\mathrm{d}{{\bm{\xi}}},{\hat{\bm{\xi}}}^i)  \\
			 \st \ \ \ & \int_{\Xi} K(\mathrm{d}{{\bm{\xi}}},{\hat{\bm{\xi}}}^i)=\frac{1}{N}, \  \forall i\in [N] \\
			& \int_{\Xi} \sum\limits_{i=1}^{N}d({{\bm{\xi}}},{\hat{\bm{\xi}}}^i) K(\mathrm{d}{{\bm{\xi}}},{\hat{\bm{\xi}}}^i)\le \epsilon_N.
		\end{aligned}
	\end{equation}
	We introduce a Lagrangian function for (\ref{eq8}), i.e., 
	\begin{equation}
		\begin{aligned}
			\label{eq9}
			L({{\bm{\xi}}},\lambda,\bm{s})= & \int_{\Xi} \sum\limits_{i=1}^{N}h(\bm{p},t,\bm{\xi})K(\mathrm{d}{\bm{\xi}},{\hat{\bm{\xi}}}^i)
			-\int_{\Xi} \sum\limits_{i=1}^{N}s_i K(\mathrm{d}{\bm{\xi}},{\hat{\bm{\xi}}}^i) 
			-\int_{\Xi} \sum\limits_{i=1}^{N}\lambda d({\bm{\xi}},{\hat{\bm{\xi}}}^i) K(\mathrm{d}{\bm{\xi}},{\hat{\bm{\xi}}}^i)
			+\frac{1}{N}\sum\limits_{i=1}^{N}{s_i}+ \lambda \epsilon_N.
		\end{aligned}
	\end{equation}
	It follows that the Lagrange dual function can be represented as
	\begin{align}
		\label{eq10}
		g(\lambda,\bm{s})  =  &\sup_{{{\bm{\xi}}}\in \Xi} L({{\bm{\xi}}},\lambda,\bm{s})  
		 = \sup_{{{\bm{\xi}}}\in \Xi} \int_{\Xi} \sum\limits_{i=1}^{N}\left(h(\bm{p},t,\bm{\xi})-s_i -\lambda  d({\bm{\xi}},{\hat{\bm{\xi}}}^i)\right) K(\mathrm{d}{\bm{\xi}},{\hat{\bm{\xi}}}^i)
		  +\frac{1}{N}\sum\limits_{i=1}^{N}{s_i}+ \lambda \epsilon_N. 	
	\end{align}
	Consequently, the dual problem of (\ref{eq8}) is given as
	\begin{subequations}
		\label{robustSAA}
		\begin{align}
			 \minimize_{\bm{s},\lambda}  \ \ \ & \frac{1}{N}\sum\limits_{i=1}^{N} s_i + \lambda\epsilon_N   \\
			\st \ \ \ & h(\bm{p},t,\bm{\xi})-\lambda d({{\bm{\xi}}},{\hat{\bm{\xi}}}^i) \le s_i, \forall {{\bm{\xi}}} \in \Xi,   i \in [N] \label{eq11b} \\
			& \lambda \ge 0 \label{eq11a}.
		\end{align}
	\end{subequations}
	
	 Consider the primal problem (\ref{eq8}) and its dual problem (\ref{robustSAA}). If $\epsilon_N > 0$, there exists a strictly feasible solution $K = F_N \times F_N$ to (\ref{eq8}). Thus, the Slater condition for their strong duality holds \citep{Shapiro2001On}. If $\epsilon_N = 0$, the Wasserstein ball $\mathcal{F}_N$ reduces to a singleton $\{F_N\}$ and  (\ref{eq8}) changes to a sample average problem $\frac{1}{N}\sum_{i=1}^{N}h(\bm{p},t,\hat{\bm{\xi}}^i)$. Indeed, any feasible solution to the dual problem (\ref{robustSAA}) satisfies that $s_i \ge h(\bm{p},t,\hat{\bm{\xi}}^i)$ with ${\bm{\xi}} = \hat{\bm{\xi}}^i$ and $s_i \ge 0$ with ${\bm{\xi}} \neq \hat{\bm{\xi}}^i$ when $\lambda$ tends to infinity. Accordingly, the optimal value of problem (\ref{robustSAA}) reduces to the sample average problem $\frac{1}{N}\sum_{i=1}^{N}h(\bm{p},t,\hat{\bm{\xi}}^i)$ as well. Overall, there is no duality gap between (\ref{eq8}) and (\ref{robustSAA}) under any case. Thus, it is sufficient to solve the dual problem (\ref{robustSAA}).
	
	Since $h(\bm{p},t,\bm{\xi}) = [{\bm{\xi}}^T\bm{p}-t]^+$, the constraint in (\ref{eq11b}) amounts to
	\begin{align}	
		&\sup_{{\bm{\xi}}\in \Xi} \ \ \left\{{\bm{\xi}}^T\bm{p}-t- \lambda  {\| {{\bm{\xi}}}-{\hat{\bm{\xi}}}^i \|}_p \right \}  \le s_i,  \label{eq12a}  \\
		&\sup_{{\bm{\xi}}\in \Xi} \ \ \left \{ - \lambda  {\| {{\bm{\xi}}}-{\hat{\bm{\xi}}}^i \|}_p \right \} \le s_i.   \label{eq12b}		
	\end{align}
	 The inequality (\ref{eq12b}) implies that $s_i \ge 0$ since $\lambda \ge 0$ and $\hat{\bm{\xi}}^i \in \Xi$. For simplicity, we denote $\Delta \bm{u}_i = {\bm{\xi}}-\hat{\bm{\xi}}^i$ and re-express the left-hand side of (\ref{eq12a}) as

	\begin{equation*}
		\begin{aligned}
			& \sup_{\Delta \bm{u}_i} \left\{\bm{p}^T(\hat{\bm{\xi}}^i+\Delta \bm{u}_i)-t-\lambda {\| \Delta \bm{u}_i \|}_p \right\} \\
			&= \sup_{\Delta \bm{u}_i}\left\{\bm{p}^T\Delta \bm{u}_i-\lambda {\| \Delta \bm{u}_i \|}_p\right\} +\bm{p}^T\hat{\bm{\xi}}^i-t \\
			&= \sup_{\Delta \bm{u}_i} \left\{ ({\| \bm{p} \|}_q - \lambda){\|\Delta \bm{u}_i\|}_p \right\} +\bm{p}^T\hat{\bm{\xi}}^i-t  \\
			& =\left\{	
			\begin{array}{ll}
				\bm{p}^T\hat{\bm{\xi}}^i -t,   &\text{if} \   {\| \bm{p} \|}_q \le \lambda  \\
				+\infty,  &\text{if} \   {\| \bm{p} \|}_q > \lambda  \  \\
			\end{array}
			\right.
		\end{aligned}
	\end{equation*}
    where the second equality follows from Lemma \ref{lem 1}. Then (\ref{eq8}) can be reformulated as
		\begin{equation}
			\begin{aligned}
				\label{eq13}
				\minimize_{\bm{s},\lambda} \ \ \ & \frac{1}{N}\sum\limits_{i=1}^{N} s_i + \lambda\epsilon_N   \\
				\st \ \ \  &\bm{p}^T\hat{\bm{\xi}}^i -t \le s_i ,\ \ s_i\ge 0,\ \  \forall i\in [ N] \\
				& {\| \bm{p} \|}_q  \le \lambda. \ \ \   \\	
			\end{aligned}
		\end{equation}
	Combining (\ref{eq13}) with (\ref{RMETT}) leads to the equivalence of the w-METT and (\ref{eqMETT}). 
	
	Consequently, the DRSP model in (\ref{eq3}) can be reformulated  as a mixed 0-1 convex problem (\ref{UDP}), which completes the proof.	
\end{proof}

Theorem \ref{thoe1} implies that the DRSP model (\ref{eq3}) is equivalent to a finite mixed 0-1 convex program. Note that the program in this paper is a discrete optimization problem which is more difficult to solve than the continuous optimization program in \citet{esfahani2018data}. Furthermore, different $l_p$-norms lead to different equivalent problems of the proposed DRSP model. For example, when we adopt $l_1$-norm or $l_2$-norm in the Wasserstein distance, the DRSP model is a 0-1 LP or 0-1 SOCP problem, respectively. See Table \ref{form} for details.

\subsection{The DRSP Model with the Support Set} \label{boundedsec}

The travel time is finite in practice and thus its support set should not be neglected. In this subsection, we incorporate this information by assuming that the travel time $\xi_{ij}$ belongs to an empirical interval $[a_{ij},b_{ij}]$ where $a_{ij} = \min_{k \in [N]}{\xi}_{ij}^k$ and $b_{ij} = \max_{k\in [N]}{\xi}_{ij}^k$. Then we derive an equivalent formulation for our DRSP model (\ref{eq3}).

Let $\Xi = [\bm{a},\bm{b}]$ in this subsection denote the support set where $\bm{a}=\{a_{ij}:(i,j)\in \mathcal{A}\}$ and $\bm{b}=\{b_{ij}:(i,j)\in \mathcal{A}\}$. Since $\Xi = [\bm{a},\bm{b}]$  is compact, any distribution $F $ in  $\mathcal{F}_N$ given by (\ref{Wset}) in this subsection automatically satisfies Assumption \ref{assum1}. Following this, we derive equivalent problems for w-METT (\ref{RMETT}) and the proposed DRSP model (\ref{eq3}).

\begin{theo}
	\label{thoe3}
	Let $\Xi = [\bm{a},\bm{b}]$, then the w-METT in (\ref{RMETT}) over the Wasserstein ball (\ref{Wset}) can be computed by a finite convex problem
	\begin{subequations}
		\label{eq14}
		\begin{align}
			 \minimize_{t,\bm{s},\lambda,\bm{\gamma}_i,\bm{\eta}_i} \ \ \ & t + \frac{1}{\alpha}\left \{\frac{1}{N}\sum\limits_{i=1}^{N} s_i + \lambda\epsilon_N \right \}  \label{bdpa}\\
			\st \ \ \ &(\bm{p}+{\bm{\gamma}}_i-{\bm{\eta}}_i)^T\hat{\bm{\xi}}^i-{\bm{\gamma}}_i^T\bm{a}+{\bm{\eta}}_i^T\bm{b} -t\le s_i \ \ \label{bdpb} \\
			&  {\| {\bm{\gamma}}_i+\bm{p}-{\bm{\eta}}_i \|}_q  \le \lambda  \label{bdpc} \\
			& {\bm{\eta}}_i\ge 0, {\bm{\gamma}}_i \ge 0, s_i\ge 0, \forall i\in [N] \label{bdpd}.
		\end{align}
	\end{subequations}
	Moreover, the DRSP problem (\ref{eq3}) is re-expressed as
	\begin{equation}
		\label{BDP}
		\begin{aligned}
			 \minimize_{\bm{p},t,\bm{s},\lambda,{\bm{\gamma}}_i,{\bm{\eta}}_i} \ \ \ & t + \frac{1}{\alpha}\left \{\frac{1}{N}\sum\limits_{i=1}^{N} s_i + \lambda\epsilon_N \right \}  \\
			 \st\ \ \ &\bm{p}\in \mathcal{P}, \ (\ref{bdpb}),\ (\ref{bdpc}) \ \text{and} \ (\ref{bdpd}).
		\end{aligned}
	\end{equation}
\end{theo}

\begin{proof}	
	The strong duality still holds for the inner largest expectation in (\ref{RMETT}), which allows us to reformulate problem (\ref{RMETT}) as
	\begin{subequations}
	\begin{align}
		 \minimize_{t,\bm{s},\lambda} \ \ \ & t+ \frac{1}{\alpha}\left\{\frac{1}{N}\sum\limits_{i=1}^{N} s_i + \lambda\epsilon_N \right\}\\
		 \st \ \ \ & h(\bm{p},t,\bm{\xi})- \lambda d({{\bm{\xi}}},{\hat{\bm{\xi}}}^i)  \le s_i, \forall {{\bm{\xi}}} \in \Xi, i\in [N]  \label{eq_fun}\\
		&  \lambda \ge 0.	
	\end{align}
	\end{subequations}
	The constraint in (\ref{eq_fun}) can be represented as
	\begin{align}
		&\sup_{{\bm{\xi}}\in \Xi} \left\{ {\bm{\xi}}^T\bm{p}-t- \lambda  {\| {\bm{\xi}}-\hat{\bm{\xi}}^i\|}_p\right\}  \le s_i \label{eq24a}\\
		&\sup_{{\bm{\xi}}\in \Xi} \left\{ - \lambda  {\| {{\bm{\xi}}}-{\hat{\bm{\xi}}}^i \|}_p\right\}  \le s_i,  \label{eq24b}	
	\end{align}
	where the inequality in (\ref{eq24b}) implies that $s_i \ge 0$.

	Note that Lemma \ref{lem 1} in Section \ref{nosupport} cannot be applied directly to (\ref{eq24a}). We utilize the Lagrange dual function to solve this problem. For brevity, we denote $\Delta \bm{u}_i = {\bm{\xi}}-\hat{\bm{\xi}}^i$ and express the Lagrangian function of $\sup_{{\bm{\xi}}\in \Xi} \left\{ {\bm{\xi}}^T\bm{p}-t- \lambda  {\| {{\bm{\xi}}}-{\hat{\bm{\xi}}}^i \|}_p \right\}$ as
	\begin{equation}
		\begin{aligned}
			\label{eq25}
			L(\Delta \bm{u}_i,{\bm{\gamma}}_i,{\bm{\eta}}_i)
			 = & (\bm{p}+{\bm{\gamma}}_i-{\bm{\eta}}_i)^T(\Delta \bm{u}_i+\hat{\bm{\xi}}^i)-\lambda \| {\Delta \bm{u}_i \|}_p 
			  -{\bm{\gamma}}_i^T\bm{a}+{\bm{\eta}}_i^T\bm{b} -t .
		\end{aligned}
	\end{equation}
	Then, the Lagrange dual function of (\ref{eq25}) is given by
	\begin{equation*}
		\begin{aligned}
			\label{eq19}
			g({\bm{\gamma}}_i) & =\sup_{\Delta u_i} L(\Delta \bm{u}_i,{\bm{\gamma}}_i,{\bm{\eta}}_i) \\
			& =\sup_{\Delta\bm{u}_i}\left\{(\bm{p}+{\bm{\gamma}}_i-{\bm{\eta}}_i)^T\Delta \bm{u}_i -\lambda \| {\Delta \bm{u}_i \|}_p\right\}+(\bm{p}+{\bm{\gamma}}_i-{\bm{\eta}}_i)^T\hat{\bm{\xi}}^i 
			-{\bm{\gamma}}_i^Ta+{\bm{\eta}}_i^Tb   -t\\
			& =\sup_{\Delta \bm{u}_i} \left\{ ({\| \bm{p}+{\bm{\gamma}}_i-{\bm{\eta}}_i \|}_q-\lambda ) \| {\Delta \bm{u}_i \|}_p\right\} +(\bm{p}+{\bm{\gamma}}_i-{\bm{\eta}}_i)^T\hat{\bm{\xi}}^i 
			 -{\bm{\gamma}}_i^T\bm{a}+{\bm{\eta}}_i^T\bm{b}  -t\\
			&=\left\{
			\begin{array}{ll}
				+\infty,&\text{if} \  {\| \bm{p}+{\bm{\gamma}}_i-{\bm{\eta}}_i \|}_q > \lambda   \\
				(\bm{p}+{\bm{\gamma}}_i-{\bm{\eta}}_i)^T\hat{\bm{\xi}}^i-{\bm{\gamma}}_i^T\bm{a}+{\bm{\eta}}_i^T\bm{b} -t, &\text{if} \  {\| \bm{p}+{\bm{\gamma}}_i-{\bm{\eta}}_i \|}_q \le \lambda
			\end{array}
			\right.
		\end{aligned}
	\end{equation*}
	where the third equality follows from Lemma \ref{lem 1}. 
	
	Consequently, $\sup_{{\bm{\xi}}\in \Xi}\left\{ \ {\bm{\xi}}^T\bm{p}-t- \lambda  {\| {{\bm{\xi}}}-{\hat{\bm{\xi}}}^i \|}_p \right\}$ admits an equivalent problem
	\begin{equation}
		\label{METTboundeq}
		\begin{aligned}
			 \minimize_{{\bm{\gamma}}_i,{\bm{\eta}}_i} \ \ \ & (\bm{p}+{\bm{\gamma}}_i-{\bm{\eta}}_i)^T\hat{\bm{\xi}}^i-{\bm{\gamma}}_i^T\bm{a}+{\bm{\eta}}_i^T\bm{b} -t \\
			 \st \ \ \  & {\| \bm{p}+{\bm{\gamma}}_i-{\bm{\eta}}_i \|}_q  \le \lambda \\
			&  {\bm{\gamma}}_i \ge 0, \ \ {\bm{\eta}}_i \ge 0, \\
		\end{aligned}
	\end{equation}
	where the strong duality holds as the uncertainty set is non-empty. 
	
	Substituting (\ref{METTboundeq}) into constraints in (\ref{eq_fun}), we obtain that the w-METT (\ref{RMETT}) can be computed by the program (\ref{eq14}) and the DRSP model eventually is given by (\ref{BDP}).
\end{proof}

We also summarize the results in Theorem \ref{thoe3} in  Table \ref{form}.  {Different from the NP-hard moment-based DRSP model with support set in \citet{Zhang2017Robust}, Table \ref{form} shows that the Wasserstein distance with $l_1$-norm, $l_2$-norm and $l_\infty$-norm leads to a tractable problem for our DRSP model, e.g., the mixed 0-1 LP or SOCP problem.}
\subsection{Asymptotic consistency}\label{Asy-cons}
We finally discuss the asymptotic consistency of our DRSP model under the following mild condition.
\begin{assum}
	\label{assum2}
	For the true distribution $F$, there exists a constant $c > 1$ such that 
	$$ \int_{\Xi}\exp(\| {\bm{\xi}} \|^c_p) F(\mathrm{d}{\bm{\xi}}) \le \infty $$	
\end{assum}
{Under Assumption \ref{assum2}, the asymptotic consistency of our model can be formalized below.}

\begin{theo}
	\label{asy_con}
	{Under Assumption \ref{assum2} and let $\beta_N \in (0,1)$ with $\sum_{N=1}^{\infty}\beta_N \le \infty$. Define the radius $\epsilon_N$ of the Wasserstein ball as 
		\begin{equation*}
		\epsilon_N(\beta_N)=\left\{\begin{array}{ll}
		\left(\frac{\log(c_1\beta_N^{-1})}{c_2N}\right)^{1/\max\{n,2\}}  & \rm{if} \ \  N \ge \frac{\log(c_1\beta_N^{-1})}{c_2}  \\
		\left(\frac{\log(c_1\beta_N^{-1})}{c_2N}\right)^{1/c} & \rm{if} \ \  N < \frac{\log(c_1\beta_N^{-1})}{c_2}, 
		\end{array}
		\right.
		\end{equation*}	
		where $c_1$ and $c_2$ are two positive constants that depend on the constant $c$ in Assumption \ref{assum2}.}
		
		{Then the optimal value and the optimal solution of the DRSP model (\ref{eq3}) converge to those of the SSP model (\ref {SSP}) with probability one as $N$ tends to infinity.}
\end{theo}
\begin{proof}
	{Since $h(\bm{p},t,\bm{\xi})$ is continuous in ${\bm{\xi}}$ and there exists $L \ge 0$ with $|~h(\bm{p},t,\bm{\xi})~| \le L(1+\| {\bm{\xi}} \|)$ for all $\bm{p} \in \mathcal{P}, t\in \mathbb{R}$ and ${\bm{\xi}} \in \Xi$, then the asymptotic consistency of our DRSP model follows from Theorem $3.6$ in \citet{esfahani2018data}.}
\end{proof}

\section{The Worst-case Distribution Achieving the w-METT} \label{worst-dis}
{In this section we derive the worst-case distribution that attains the supremum of the w-METT in \eqref{RMETT} for any feasible pair of $\{\bm{p},t\}$, i.e., $\bm{p}\in\mathcal{P}$ and $t\in\mathbb{R}$.}

\begin{lem}
	{\label{thm_dis}Given any feasible pair of $\{\bm{p},t\}$, then $\sup_{F\in \mathcal{F}_N}\mathbb{E}_F \{h(\bm{p},t,\bm{\xi})\}$  is equivalent to 
	\begin{equation}
		\label{worst-case-B}
		\sup_{\tilde{\bm{\xi}}\in \mathcal{B}}\left\{\frac{1}{N}\sum_{i=1}^{N}h(\bm{p},t,\bm{\xi}^{(i)})\right\}
	\end{equation}
	where 
	\begin{equation*}
		\label{setB}
		\mathcal{B} = \left\{(\xi^{(1)},\dots,\xi^{(N)})~|~ \frac{1}{N}\sum_{i=1}^{N}d(\bm{\xi^{(i)}},\hat{\bm{\xi}}^{i})\le\epsilon_N,\  \xi^{(i)}\in \Xi\right\}.
	\end{equation*}}
\end{lem}

\begin{proof}
	{Fix any solution $\{\bm{p},t\}$, it follows from the weak duality that
	\begin{equation*}
		\begin{aligned}
			\sup_{\tilde{\bm{\xi}}\in \mathcal{B}}\left\{\frac{1}{N}\sum_{i=1}^{N}h(\bm{p},t,\bm{\xi}^{(i)})\right\} &\le \inf_{\lambda\ge 0}\sup_{\bm{\xi}^{(i)}\in \Xi}\left\{\frac{1}{N}\sum_{i=1}^{N}h(\bm{p},t,\bm{\xi}^{(i)})-\lambda\left(\frac{1}{N}\sum_{i=1}^{N}d(\bm{\xi^{(i)}},\hat{\bm{\xi}}^{i})-\epsilon_N\right)\right\}\\
			& = \inf_{\lambda \ge 0}\left\{\lambda\epsilon_N + \frac{1}{N}\sum_{i=1}^{N}\sup_{\bm{\xi}^{(i)}\in \Xi}\left\{h(\bm{p},t,\bm{\xi}^{(i)})-\lambda d(\bm{\xi^{(i)}},\hat{\bm{\xi}}^{i}) \right\}\right\}\\
			& = \sup_{F\in \mathcal{F}_N}\mathbb{E}_F \left\{h(\bm{p},t,\bm{\xi})\right\}
		\end{aligned}	
	\end{equation*}
	where the last equality follows from the program (\ref{robustSAA}) in the proof of Theorem \ref{thoe1}. This implies that $\sup_{F\in \mathcal{F}_N}\mathbb{E}_F \{h(\bm{p},t,\bm{\xi})\}$ is greater than (\ref{worst-case-B}).}
	
	{Next we show that $\sup_{F\in \mathcal{F}_N}\mathbb{E}_F \{h(\bm{p},t,\bm{\xi})\}$ is also less than (\ref{worst-case-B}). For any $\varepsilon \ge 0$, it follows from the equivalent program (\ref{robustSAA}) that there exists $\{\tilde{\bm{\xi}}^{(i)}\}_{i\in[N]}\subseteq \Xi$ such that
	\begin{equation}\label{contradict}
		\begin{aligned}
			\sup_{F\in \mathcal{F}_N}\mathbb{E}_F \{h(\bm{p},t,\bm{\xi})\}-\varepsilon < \inf_{\lambda \ge 0}\left\{\lambda\epsilon_N + \frac{1}{N}\sum_{i=1}^{N}\left\{h(\bm{p},t,\tilde{\bm{\xi}}^{(i)})-\lambda d(\tilde{\bm{\xi}}^{(i)},\hat{\bm{\xi}}^{i})\right\} \right\}	.
		\end{aligned}
	\end{equation}
	If $\left(\tilde{\bm{\xi}}^{(1)},\dots,\tilde{\bm{\xi}}^{(N)}\right) \notin \mathcal{B}$ and let $\lambda > 0 $, it holds that  
	\begin{equation*}
		\lambda\left\{\epsilon_N-\frac{1}{N} \sum_{i=1}^{N}d(\tilde{\bm{\xi}}^{(i)},\hat{\bm{\xi}}^{i})\right\} < 0.
	\end{equation*}
	If $\lambda$ tends to $+\infty$ in \eqref{contradict}, it leads to $\sup_{F\in \mathcal{F}_N}\mathbb{E}_F \{h(\bm{p},t,\bm{\xi})\} = -\infty$. This contradicts the fact that $\sup_{F\in \mathcal{F}_N}\mathbb{E}_F \{h(\bm{p},t,\bm{\xi})\} \ge \mathbb{E}_{F_N} \{h(\bm{p},t,\bm{\xi})\} > -\infty$. Thus, $\left(\tilde{\bm{\xi}}^{(1)},\dots,\tilde{\bm{\xi}}^{(N)}\right) \in \mathcal{B}$ and 
	\begin{equation*}
		\begin{aligned}
			\sup_{F\in \mathcal{F}_N}\mathbb{E}_F \{h(\bm{p},t,\bm{\xi})\}-\varepsilon < \inf_{\lambda\ge 0}\sup_{\tilde{\bm{\xi}}\in \mathcal{B}}\left\{\lambda\epsilon_N + \frac{1}{N}\sum_{i=1}^{N}\left\{h(\bm{p},t,\tilde{\bm{\xi}}^{(i)})-\lambda d(\tilde{\bm{\xi}}^{(i)},\hat{\bm{\xi}}^{i})\right\} \right\}
		\end{aligned}
	\end{equation*}
	Since for any $\left({\bm{\xi}}^{(1)},\dots,{\bm{\xi}}^{(N)}\right) \in \mathcal{B}$ it holds that 
	\begin{equation*}
		\begin{aligned}
			\lambda\left\{\epsilon_N- \frac{1}{N}\sum_{i=1}^{N}d({\bm{\xi}}^{(i)},\hat{\bm{\xi}}^{i})\right\} \ge 0,
		\end{aligned}
	\end{equation*}
	it implies that
	\begin{equation*}
		\begin{aligned}
			\sup_{F\in \mathcal{F}_N}\mathbb{E}_F \{h(\bm{p},t,\bm{\xi})\}-\varepsilon < \sup_{\tilde{\bm{\xi}}\in \mathcal{B}}\left\{\frac{1}{N}\sum_{i=1}^{N}\left\{h(\bm{p},t,{\bm{\xi}}^{(i)})\right\} \right\}
		\end{aligned}
	\end{equation*}
	Letting $\varepsilon$ tend to zero leads to the desired inequality and hence we obtain that 
	\begin{equation*}
		\begin{aligned}
			\sup_{F\in \mathcal{F}_N}\mathbb{E}_F \{h(\bm{p},t,\bm{\xi})=\sup_{\tilde{\bm{\xi}}\in \mathcal{B}}\left\{\frac{1}{N}\sum_{i=1}^{N}h(\bm{p},t,\bm{\xi}^{(i)})\right\}.
		\end{aligned}
	\end{equation*}}
\end{proof}
{Note that the objective function of (\ref{worst-case-B}) is continuous and its feasible set $\mathcal{B}$ is compact, the  supremum problem admits an optimal solution. That is,  ``sup" can be replaced by ``max". Since Lemma \ref{thm_dis} implies the equivalence of the optimal value of (\ref{worst-case-B}) and $\sup_{F\in \mathcal{F}_N}\mathbb{E}_F \{h(\bm{p},t,\bm{\xi})\}$ for any feasible pair of $\{\bm{p},t\}$, we can construct a worst-case distribution based on the optimal solution of (\ref{worst-case-B}).}

\begin{theo}
	{\label{worst-case-dis-1}
	Given any feasible pair of $\{\bm{p},t\}$, let $\bm{\xi}_{\{\bm{p},t\}}=\left(\bm{\xi}^{(1)}_{\{\bm{p},t\}},\dots,\bm{\xi}^{(N)}_{\{\bm{p},t\}}\right)$ be an optimal solution of the optimization problem in Lemma \ref{thm_dis}. Then the following distribution 
	\begin{equation*}
		\begin{aligned}
			F^*_{\{\bm{p},t\}} = \frac{1}{N}\sum_{i=1}^{N}\delta_{\bm{\xi}^{(i)}_{\{\bm{p},t\}}}
		\end{aligned}
	\end{equation*}
	is the worst-case distribution, i.e., 
	\begin{equation*}
		\begin{aligned}
			\sup_{F\in \mathcal{F}_N}\mathbb{E}_F\left\{h(\bm{p},t,{\bm{\xi}})\right\} = \mathbb{E}_{F^*_{\{\bm{p},t\}}}\left\{h(\bm{p},t,{\bm{\xi}})\right\}.
		\end{aligned}
	\end{equation*}}
\end{theo}
\begin{proof}
	{We first show $F^*_{\{\bm{p},t\}}$ belongs to the Wasserstein ball $\mathcal{F}_N$. Denote a joint distribution of  $F_N$ and $F^*_{\{\bm{p},t\}}$ as
	\begin{equation*}
		\begin{aligned}
			\Pi_{\{\bm{p},t\}} = \frac{1}{N}\sum_{i=1}^{N}\delta_{(\bm{\xi}^{(i)}_{\{\bm{p},t\}},\hat{\bm{\xi}}_i)}.
		\end{aligned}
	\end{equation*}
	Then by the definition of the Wasserstein metric we obtain that
	\begin{equation*}
		\begin{aligned}
			d_W(F_N,F^*_{\{\bm{p},t\}}) \le \int\left\|{\bm{\xi}}-{\bm{\xi}}^{\prime}\right\|_p \Pi_{\{\bm{p},t\}}\left(\mathrm{d} {\bm{\xi}}, \mathrm{d} {\bm{\xi}}^{\prime}\right)=\frac{1}{N} \sum_{i=1}^{N} \| {\bm{\xi}}_{\{\bm{p},t\}}^{(i)}-\hat{\bm{\xi}}^{i}\|_p \le \epsilon_N,
		\end{aligned}
	\end{equation*} 
	where the last inequality holds since
	$\left(\bm{\xi}^{(1)}_{\{\bm{p},t\}},\dots,\bm{\xi}^{(N)}_{\{\bm{p},t\}}\right) \in \mathcal{B}$. Thus, $F^*_{\{\bm{p},t\}}$ is contained in the Wasserstein ball $\mathcal{F}_N$. Consequently, we obtain that 
	\begin{equation*}
		\begin{aligned}
			&\sup_{F\in \mathcal{F}_N}\mathbb{E}_F \left\{h(\bm{p},t,{\bm{\xi}})\right\}\ge\mathbb{E}_{F^*_{\{\bm{p},t\}}}\left\{h(\bm{p},t,{\bm{\xi}})\right\}=
			\frac{1}{N}\sum\limits_{i=1}^{N}\left(h(\bm{p},t,
			\bm{\xi}^{(i)}_{\{\bm{p},t\}})\right)
			=\sup_{F\in \mathcal{F}_N}\mathbb{E}_F \left\{h(\bm{p},t,\bm{\xi})\right\}
		\end{aligned}
	\end{equation*}
	where the last equality holds due to Lemma \ref{thm_dis}. This implies that $F^*_{\{\bm{p},t\}}$ is one of the worst-case distribution in the Wasserstein ball $\mathcal{F}_N$. Here we complete this proof.}
\end{proof}

{Theorem \ref{worst-case-dis-1} implies that the worst-case distribution has a finite support and the number of support elements is the same as the number of the data samples. As the moment-based DRSP model is only solved by an approximation method, they are unable to provide any result on the worst-case distribution.}

\section{Extensions} \label{EtOP}
\subsection{DR Bi-criteria Shortest Path Problem}
{Now we extend our DRSP model to the bi-criteria DR shortest path problem, i.e., }
{	
\begin{equation}
\label{bi-cri}
\begin{aligned}
\minimize_{ t\in \mathbb{R},\bm{p}\in \mathcal{P}}\ \ \begin{bmatrix}
\sum\limits_{(i,j)\in \mathcal{A}}c_{ij}p_{ij}\\
t + \frac{1}{\alpha}\sup_{F\in\mathcal{F}_N}\mathbb{E}_F\{h({\bm{\xi}},t,\bm{p})\}
\end{bmatrix},
\end{aligned}
\end{equation}
where $\bm{c} = \{c_{ij}:(i,j)\in \mathcal{A}\}$ is the deterministic vector of the travel cost over all arcs and 
$c_{ij}$ is the cost on each arc $(i,j) \in \mathcal{A}$. Moreover, we set $\Xi = \mathbb{R}^n$ in this subsection.}

{Since the existence of a path simultaneously minimizing both objectives in (\ref{bi-cri}) cannot be guaranteed, we follow the idea of \citet{mangasarian1994nonlinear} to seek weakly robust efficient paths with a bearable trade off between two objectives.}

\begin{defi}
	\label{scalar}
	{Let $f_1(\bm{p}) = \bm{c}^T\bm{p}$, $f_2(\bm{p},{t}) = t + \frac{1}{\alpha}\mathbb{E}_F \{h({\bm{\xi}},t,\bm{p})\}$, then $\left\{\overline{\bm{p}},\overline{t}\right\}$ is a weakly robust efficient solution of (\ref{bi-cri}) if and only if there exists $\overline{\lambda}_1 \ge 0,\overline{\lambda}_2 \ge 0$ and $\overline{\lambda}_1+\overline{\lambda}_2\neq 0$ such that 
	\begin{equation*}
	\begin{aligned}
	&\overline{\lambda}_1f_1(\bm{p})+\overline{\lambda}_2\sup_{F\in \mathcal{F}_N}f_2(\bm{p},{t})
	\ge \overline{\lambda}_1f_1(\overline{\bm{p}})+\overline{\lambda}_2\sup_{F\in \mathcal{F}_N}f_2(\overline{\bm{p}},\overline{t}) 
	\end{aligned}
	\end{equation*}
	holds for any feasible pair of $\left\{{\bm{p}},{t}\right\}$.}
\end{defi}

\begin{coro}
	\label{b-DRSP} 
	{The weakly robust efficient solution to \eqref{bi-cri} over the Wasserstein ball $\mathcal{F}_N$ can be obtained by solving the following problem
	\begin{equation}\label{db-DRSP}
	\begin{aligned}
	\minimize_{\bm{p},\bm{s},\lambda} \ \ \ &\overline{\lambda}_1f_1(\bm{p})+\overline{\lambda}_2\left(t+\frac{1}{\alpha}\frac{1}{N}\sum\limits_{i=1}^{N}s_i+\lambda\epsilon_N\right) \\
	\st \ \ &\bm{p}^T\hat{\bm{\xi}}^i-t \le s_i, s_i \ge 0, \forall i\in [N],\\
	& \| \bm{p} \|_q \le \lambda, \bm{p}\in \mathcal{P}.			
	\end{aligned}
	\end{equation}}
\end{coro}

\begin{proof}
		{Since the uncertainty only arises in the second objective function $f_2(\bm{p})$,  which equals the objective function in our DRSP model, the result follows from Theorem \ref{thoe1}.}
\end{proof}

{Corollary \ref{b-DRSP} implies that our DRSP model can be extended to the DR bi-criteria shortest path problem and its weakly robust efficient solutions can be obtained simply by solving a deterministic convex problem.}

\subsection{DR Minimum Cost Flow Problem}
{This subsection considers the DR minimum cost flow problem based on the Wasserstein ball, i.e.,
\begin{equation}
\begin{aligned}
\label{DRMCF}
\minimize_{\bm{x}} \ \ \ &\sup_{F\in \mathcal{F}_N}\mathbb{E}_F\left\{\sum\limits_{(i,j)\in\mathcal{A}}c_{ij}x_{ij}\right\}\\
\st \ \ & \sum\limits_{j:(i,j)\in\mathcal{A}}x_{ij}-\sum\limits_{j:(j,i)\in\mathcal{A}}x_{ji}=b_i,\forall i \in \mathcal{V}, \\
& 0 \le x_{ij} \le u_{ij}, \ \ \forall (i,j) \in \mathcal{A}, \\
\end{aligned}
\end{equation}
where $c_{ij}$ and $u_{ij} \ge 0$ are the random cost and the capacity of  arc $(i,j)$ respectively.  $b_i > 0$ is the supply of each vertex $i$. Let $\bm{x} = \{x_{ij}:(i,j)\in\mathcal{A}\}$ represent a flow from the source vertex to the sink vertex, where $x_{ij}$ is the flow on arc $(i,j)$. Let $\mathcal{X}$ be the set of feasible flows of problem (\ref{DRMCF}). }

{We are interested in the problem whose distribution of the random cost vector $\bm{c} = \{c_{ij}:(i,j)\in \mathcal{A}\} \in \mathbb{R}^n$ belongs to a Wasserstein ball $\mathcal{F}_N$. 	
We prove that (\ref{DRMCF}) can be reformulated as a modified network flow problem.}

\begin{coro}
	\label{MMCF}
	{Let $d(\cdot)  = \| \cdot \|_\infty$ and $\Xi = \mathbb{R}^n$ in the Wasserstein ball $\mathcal{F}_N$, the optimal DR flow to \eqref{DRMCF} can be obtained by solving 
	\begin{equation}
	\begin{aligned}
	\label{E2-DRMCF}
	\minimize_{\bm{x}} \ \ \ &\sum\limits_{(i,j)\in \mathcal{A}}\left(\frac{1}{N}\sum\limits_{k=1}^{N}\hat{c}_{ij}^k+\epsilon_N\right)x_{ij} \\
	\st \ \  & \bm{x} \in \mathcal{X},
	\end{aligned}
	\end{equation}
	where $\hat{\bm{c}}^k$ is the $k$th sample of the random cost vector and $\mathcal{X}$ is the set of all feasible flows.}
\end{coro}

\begin{proof}
	{By Theorem \ref{thoe1}, the problem (\ref{DRMCF}) admits an equivalent program
	\begin{equation}
	\begin{aligned}
	\label{E-DRMCF}
	\minimize_{\bm{x}} \ \ \ &\frac{1}{N}\sum\limits_{k=1}^{N}s_k + \lambda\epsilon_N \\
	\st \ \ & \sum\limits_{(i,j)\in \mathcal{A}}\hat{c}_{ij}^kx_{ij}\le s_k, \forall k \in [N],\\
	& \lambda \ge \| \bm{x} \|_1, \bm{x} \in \mathcal{X}.
	\end{aligned}
	\end{equation}
	Eliminating the variables $s_k$ and $\lambda$ from (\ref{E-DRMCF}) leads to
	\begin{equation}
	\begin{aligned}
	\label{E1-DRMCF}
	\minimize_{\bm{x}} \ \ \ &\frac{1}{N}\sum\limits_{k=1}^{N}\sum\limits_{(i,j)\in \mathcal{A}}\hat{c}_{ij}^kx_{ij} + \epsilon_N \| \bm{x} \|_1 \\
	\st \ \  & \bm{x} \in \mathcal{X}.
	\end{aligned}
	\end{equation}
	Since $x_{ij}\ge 0$, we can rewrite $\| \bm{x} \|_1$ as $\sum_{(i,j)\in \mathcal{A}}x_{ij}$ and reformulate (\ref{E1-DRMCF}) as (\ref{E2-DRMCF}),	which is a nominal minimum cost flow problem with edge cost $c_{ij} =\frac{1}{N}\sum_{k=1}^{N}\hat{c}_{ij}^k+\epsilon_N$.}
\end{proof}

{Corollary \ref{MMCF} implies that the DR minimum cost flow problem over the Wasserstein ball can be converted to a deterministic network flow problem with arc cost $c_{ij} = \frac{1}{N}\sum_{k=1}^{N}\hat{c}_{ij}^k+\epsilon_N$.}

\section{Experiments} \label{exper}
{Numerical experiments are conducted to validate the performance of our DRSP model in this section. {We take the commonly-used $l_1$-norm to compute the difference of travel time of two vectors. In view of Table \ref{form}, our DRSP model is reduced to a mixed 0-1 LP problem, which can be solved via existing optimization techniques.} All experiments are implemented on a 64 bit PC with an Intel Core i5-7500 CPU at 3.4GHz and 8 GB RAM. Cplex 12.6 is used to solve the mixed 0-1 LP problem.}

\subsection{DR Shortest Path Problems}
\label{DR-SPP}
{Experiments on the Eastern Massachusetts (EMA) network \citep{Transportation2019} with 74 vertices and 258 links are firstly performed to reveal the impact of different parameters on our DRSP model. Travel time of each arc is captured by a random vector ${\bm{\xi}}=\{\xi_{ij}:(i,j)\in\mathcal{A}\}$. We find paths from the origin vertex $1$ to the destination vertex $74$ by solving our DRSP model.}

\begin{figure}[htbp]
	\centering
	\includegraphics[scale=0.4]{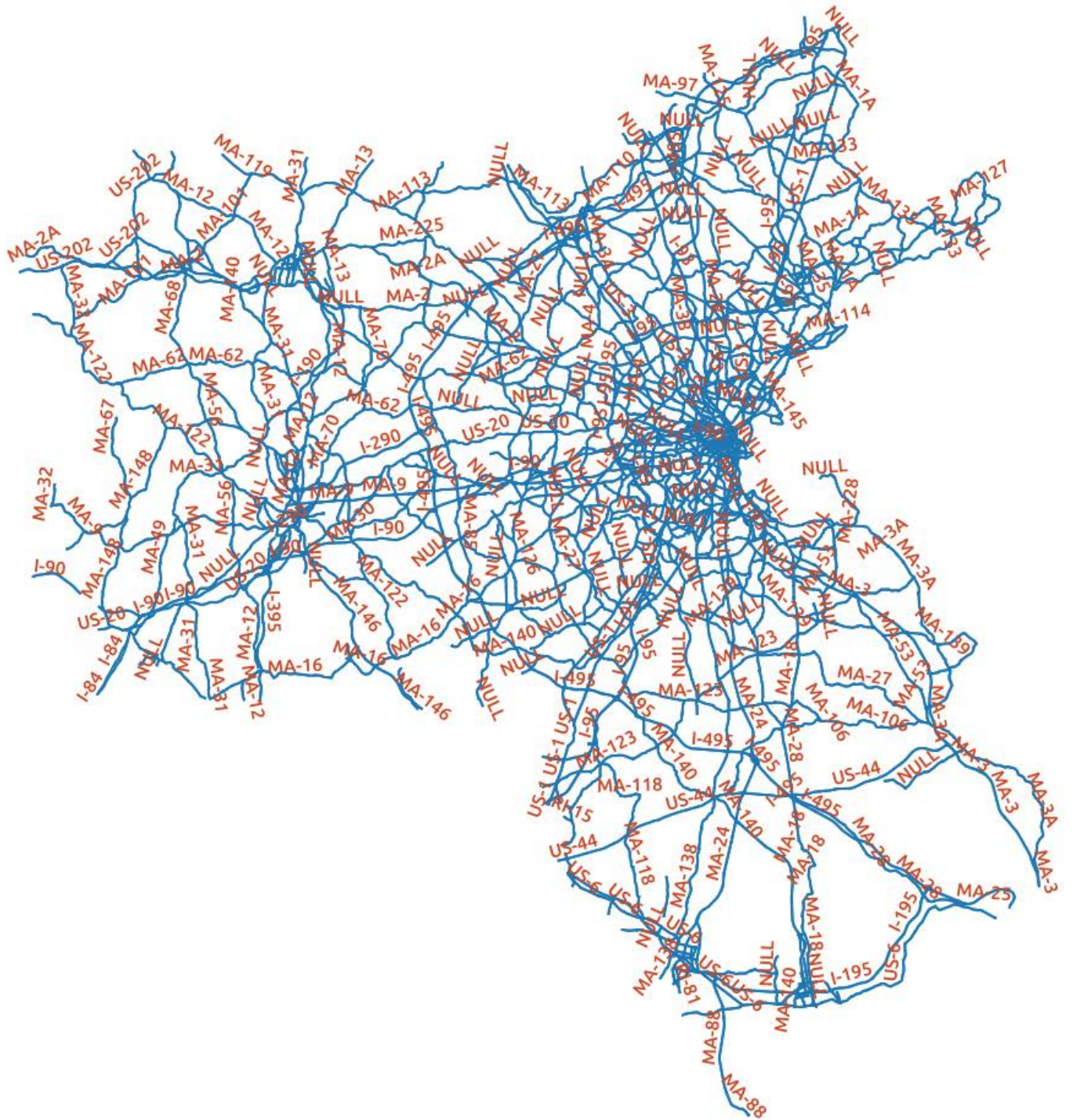}
	\caption{Eastern Massachusetts (EMA) network with 74 vertices and 258 links.}
	\label{fig:topo}
\end{figure}

{We assume that $\xi_{ij}$ of different arcs is independent across the network and follows a mixture of Gaussian distribution $\mathcal{N}(\mu_{ij},{\mu_{ij}}\times 10)$ and uniform distribution $\mathcal{U}(0,\mu_{ij})$ where $\mu_{ij}$ is obtained from \citet{Transportation2019}.  Clearly, both distributions satisfy Assumption \ref{assum1} with finite first moments. In the experiment, half of the dataset are generated from $\mathcal{N}(\mu_{ij},{\mu_{ij}}\times 10)$ and the rest are generated from $\mathcal{U}(0,\mu_{ij})$. Impacts of the Wasserstein radius $\epsilon_N$ and the dataset size  $N$ on the out-of-sample performance of our DRSP model are tested respectively. {We compare our model with the moment-based DRSP model  in \citet{Zhang2017Robust} and the SAA method in terms of the out-of-sample performance and computation complexity.} Table \ref{paraset} shows the use of parameters in different experiments.}
 
{We evaluate the out-of-sample performance by  examining the cost of the model under {\em new} samples, i.e., 
 \begin{equation}
 \label{Out_P}
 \minimize_{t\in\mathbb{R}} \left\{ t + \frac{1}{\alpha }\mathbb{E}_{F}\{h(\bm{p},t,\bm{\xi})\}\right\}
 \end{equation}}
 
{Noting that the  true distribution $F$  is a mixture of the Gaussian distribution and the uniform distribution, and importantly unknown. We are unable to exactly compute \eqref{Out_P}. Instead, we randomly generate $500$ samples from the Gaussian distribution and $500$ samples from the uniform distribution as test samples to approximate it, i.e.,
 \begin{equation*}
 \minimize_{t\in\mathbb{R}} \left\{ t + \frac{1}{\alpha N_T}\sum\limits_{i=1}^{N_T}\{h(\bm{p},t,\hat{\bm{\xi}}^i)\right\}
 \end{equation*}
where $\hat{\bm{\xi}}^i$ is the $i $th test sample and $N_T$ is the size of test dataset.}

\renewcommand{\arraystretch}{1.2} 
\begin{table*}[htb]	
	\centering
	{\begin{tabular}{|l|c|c|c|}
		\hline
		{ } &$\alpha$&$\epsilon_N$&$N$\cr\hline
		{ Impact of $\epsilon_N$}&$0.1$&$\{0,0.001,0.005,0.01,0.05,0.1,0.2,\dots, 1\}$&$\{30,100,300\}$\cr\hline
		{ Impact of $N$}&$0.1$&$0.1$&$\{30,50,\dots,290\}$\cr\hline
	\end{tabular}}
	\caption{Parameters in different experiments.}
	\label{paraset}
\end{table*}

\subsection{Performance of the DRSP Model Without Support Set}\label{PerB}
{In this subsection we evaluate the performance of our DRSP model in \eqref{eq3}.}

\subsubsection{Impact of the Wasserstein Radius} \label{ImpW}
{We first conduct experiments to test the impact of the Wasserstein radius $\epsilon_N$ on the out-of-sample performance of our DRSP model.  Parameters $\alpha$, $\epsilon_N$ and $N$ are selected as the first row of Table \ref{paraset}. } 

{Since $\|\bm{p}\|_{\infty}=1$ for any feasible path ${\bm p}$, one can argue from \eqref{uc} that $l_{1}$-norm in the Wasserstein distance is not sensible to our DRSP model without the support set. In this subsection we take $l_2$-norm in the Wasserstein distance.}

{We perform $200$ independent experiments and the averaged out-of-sample performance is shown in Figure \ref{radius}. It reveals that the performance improves until the Wasserstein radius exceeds a certain value, and then deteriorates as the radius increases. As a result, the Wasserstein radius of our DRSP model should be selected based on the size of the sample dataset for good out-of-sample performance.}

\begin{figure*}[htbp]
	\centering
	\subfigure[ ]{
		\label{fig:30 }
		\includegraphics[width=0.31 \textwidth]{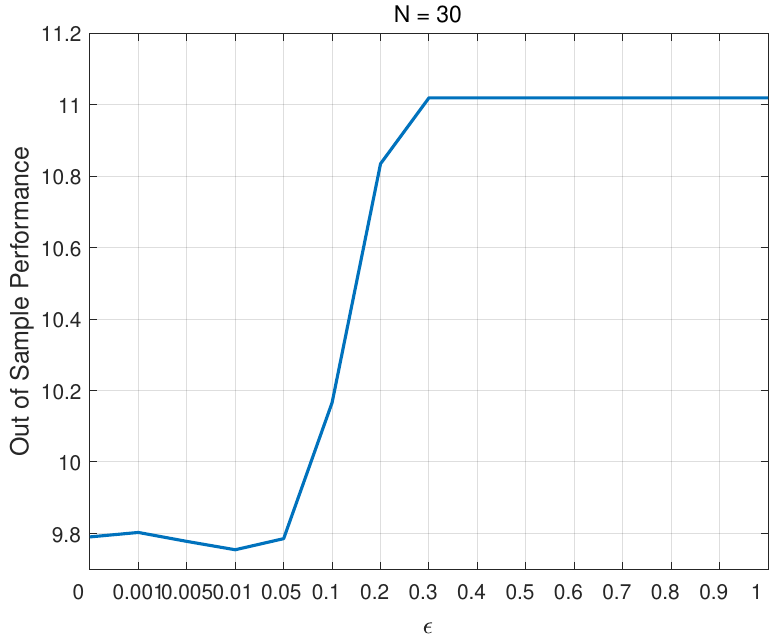}
	}
	\subfigure[ ]{
		\label{fig:100 }
		\includegraphics[width=0.31 \textwidth]{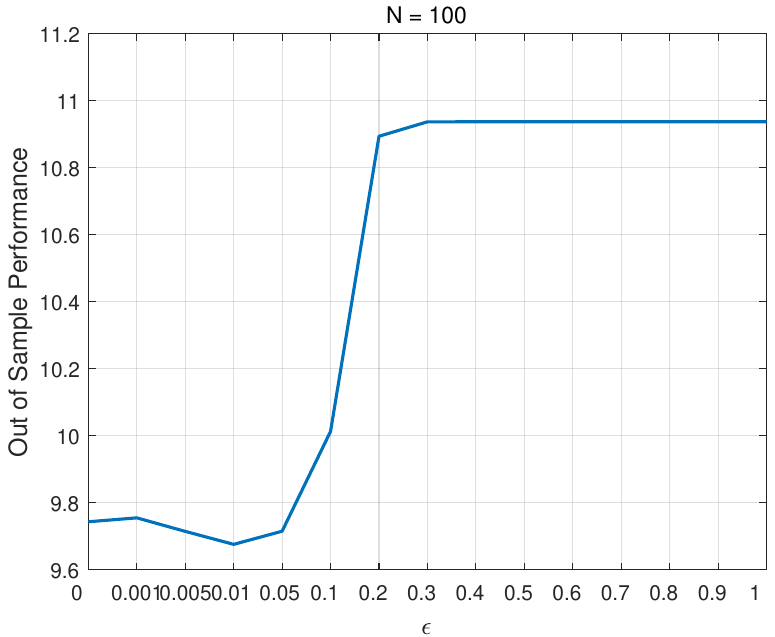}
	}
	\subfigure[ ]{
		\label{fig:300}
		\includegraphics[width=0.31 \textwidth]{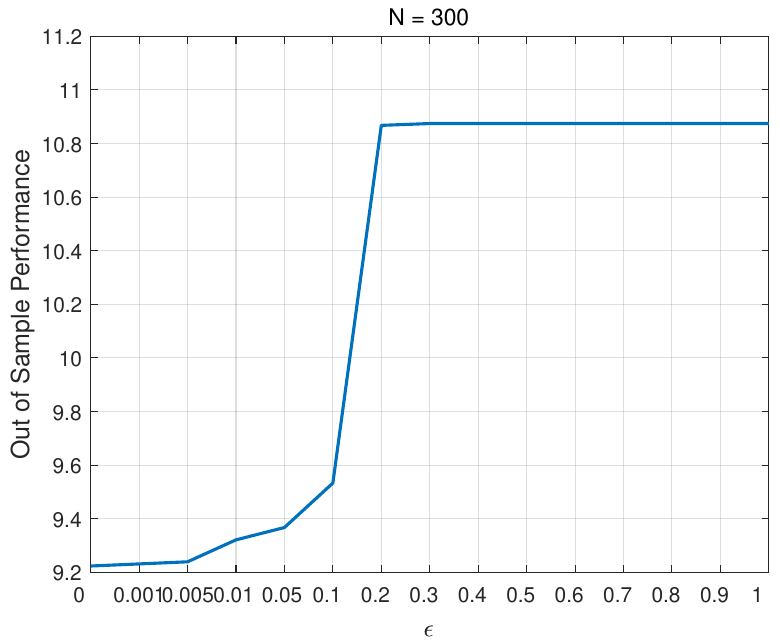}
	}
	\caption{{Averaged out-of-sample performance as a function of Wasserstein radius under sample datasets with different sizes. (a) $N=30$, (b) $N=100$, (c) $N = 300$.}}
	\label{radius}
\end{figure*}

\subsubsection{Impact of the Sample Size}\label{ImpSam}
{In this subsection we perform experiments on the sample dataset of different sizes to examine the impact of the sample sizes. Parameters in these experiments are set as the second row in Table \ref{paraset}. The averaged out-of-sample performance over $200$ independent experiments is presented in Figure \ref{fig:Sam}.}

\begin{figure}[htbp]
	\centering
	\includegraphics[scale=0.5]{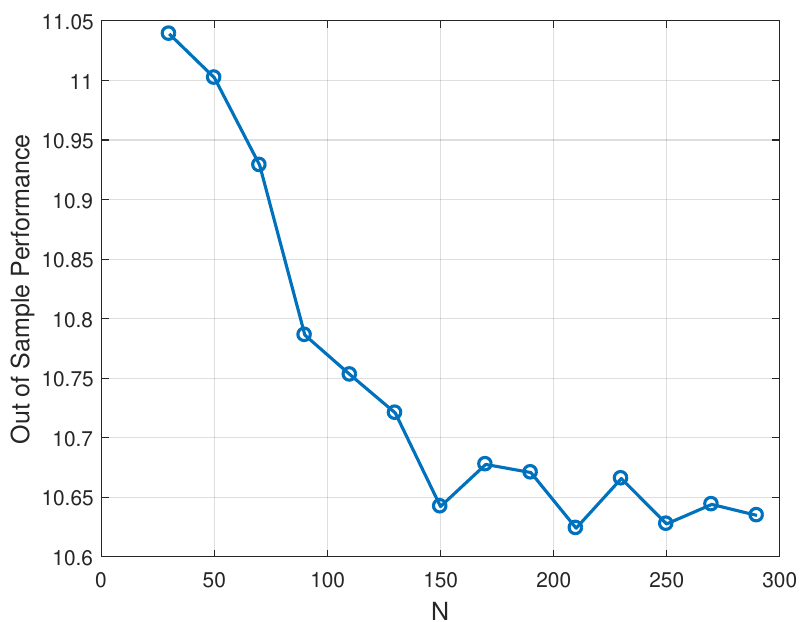}
	\caption{Averaged out-of-sample performance as a function of sample size $N$ for $200$ independent experiments,  where the support set of travel time is $\mathbb{R}^n$.}
	\label{fig:Sam}
\end{figure}

{The improvement of the out-of-performance with increasing dataset sizes in Figure \ref{fig:Sam} validates the asymptotic consistency of our model as shown in Theorem \ref{asy_con}.}

\subsection{Performance of the DRSP Model With Support Set} \label{PerS}
{This subsection validates our DRSP model with the support set $\Xi = [\bm{a},\bm{b}]$ in Section \ref{boundedsec}. We test the impact of radius $\epsilon_N$ and the sample size $N$ on our DRSP model. All parameters are set as Table \ref{paraset}. }

\subsubsection{Impact of the Wasserstein Radius}
{We first test the impact of radius $\epsilon_N$. Figure \ref{radius_ab} describes the averaged out-of-sample performance over $200$ independent experiments under different Wasserstein radii $\epsilon_N$. Similar to results in Section \ref{ImpW}, the out-of-sample performance obtains its optimal value at a certain point and then deteriorates as the radius increases. }

\begin{figure*}[htbp]
	\centering
	\subfigure[ ]{
		\label{fig:30b}
		\includegraphics[width=0.31 \textwidth]{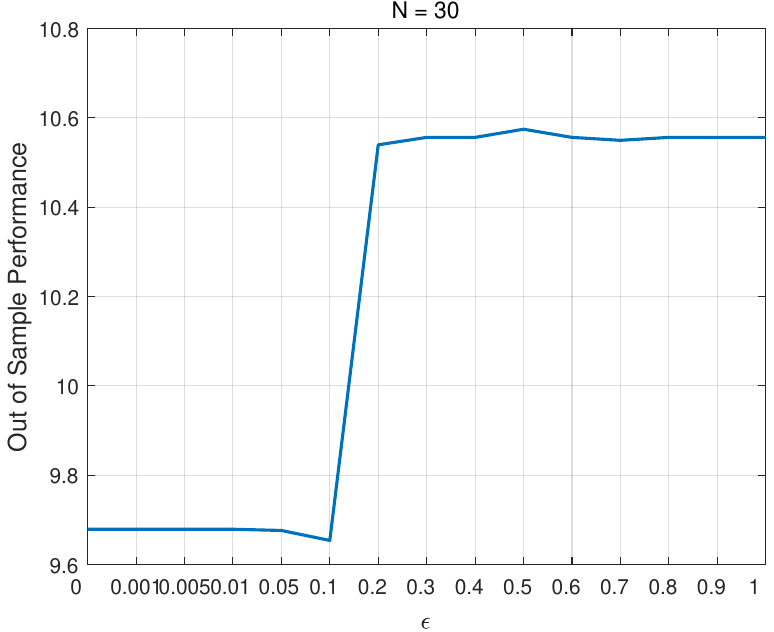} 
	}
	\subfigure[ ]{
		\label{fig:subfig: }
		\includegraphics[width=0.31 \textwidth]{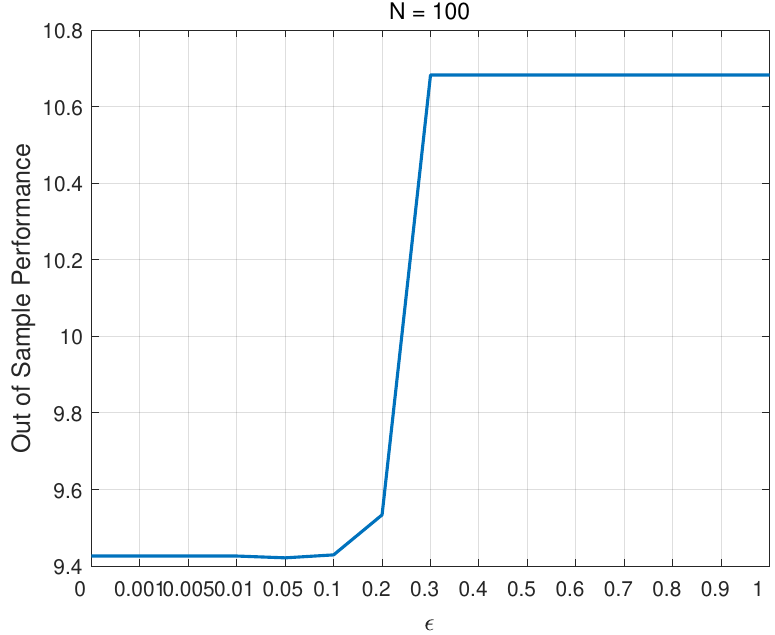}
	}
	\subfigure[ ]{
		\label{fig:300b}
		\includegraphics[width=0.31 \textwidth]{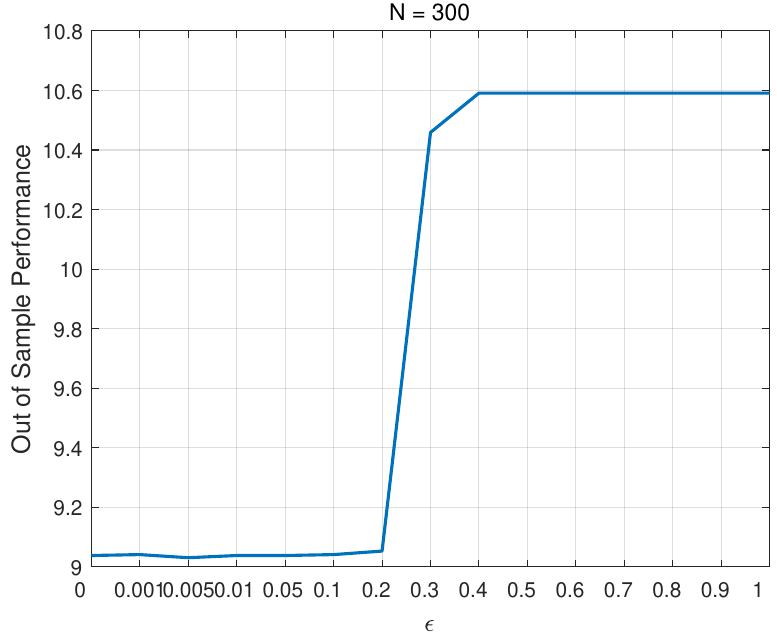}
	}
	\caption{{Averaged out-of-sample performance as a function of Wasserstein radius under the sample dataset with different sizes where the support set of travel time is $ [\bm{a},\bm{b}]$. (a) $N=30$, (b) $N=100$, (c) $N = 300$.}}
	\label{radius_ab}
\end{figure*}

\subsubsection{Impact of the Sample Size}
{Now, we test the impact of the sample size $N$. Figure \ref{fig:Sam_ab} shows the out-of-sample performance averaged over $200$ independent simulation runs as a function of $N$, in which the performance improves as $N$ increases.}

\begin{figure}[htbp]
	\centering
	\includegraphics[scale=0.5]{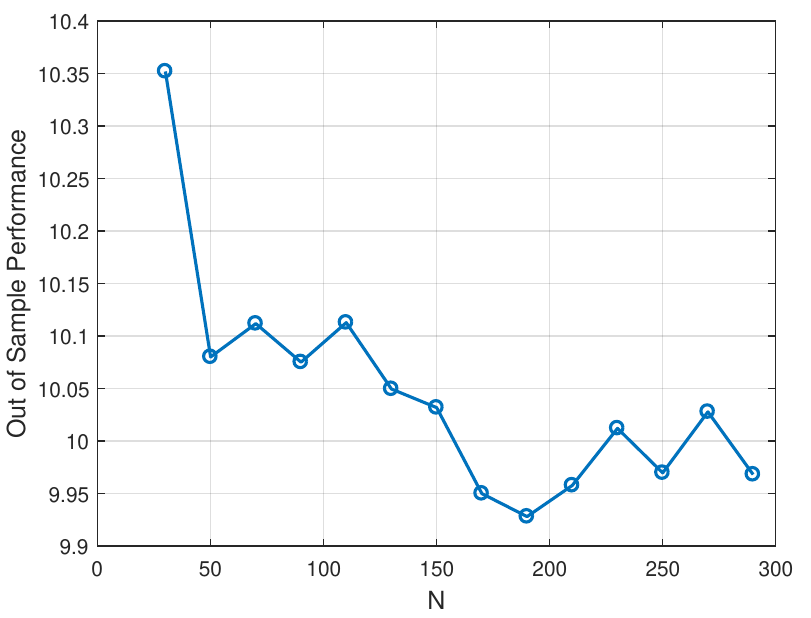}
	\caption{Averaged out-of-sample performance as a function of sample size $N$ averaged over $200$ independent simulations, where the support set of travel time is $[\bm{a},\bm{b}]$ .}
	\label{fig:Sam_ab}
\end{figure}

\subsection{Comparisons with the state-of-the-art methods}
\label{comparison}
{In this subsection we compare our model with the SAA method and the DRSP model in \citet{Zhang2017Robust} where the ambiguity set is based on moment constraints. We set $\alpha = 0.1$ and $N=\{20,30,50,100,200,300,500\}$ in this subsection. Noting that the Wasserstein radius $\epsilon_N$ depends on the training data. We tune it for different sample datasets to provide a powerful out-of-sample performance guarantee.}
{Since the moment-based model is an intractable co-positive program problem, \citet{Zhang2017Robust} derives dual approximation methods to provide a lower bound and an upper bound for the DR shortest path problem. They are denoted as M-LB and M-UB for short respectively.}

{We solve these models and evaluate their out-of-sample performance.  For the purpose of comparison, we use the following percentage difference 
\begin{equation*}
\begin{aligned}
	\left(\frac{\text{DR}}{\text{SAA}}-1\right)\times 100\%
\end{aligned}
\end{equation*} 
where {DR} denotes the out-of-sample performance of paths obtained from the DR shortest path model and SAA denotes that of the SAA method.}

{Comparisons in terms of the out-of-sample performance and the computation time are given in Table \ref{Com_in_out1} and Table \ref{Com_in_time1}, respectively. A negative value in Table \ref{Com_in_out1} indicates that the DR  performs better than the SAA. One can observe that our DRSP model with the support set (denoted as DRSP-S)  exhibits the best out-of-sample performance if the dataset is small, which however needs a longer computational time. As the dataset increases, our DRSP model without the support set also performs better than the SAA method and the moment-based DRSP model. More importantly, it can be solved in an appropriate time even when the sample set is large. Kindly note that both the exact mean and the variance are needed and essential to the moment-based DRSP model. From this perspective, they need more exact model information.}

\renewcommand{\arraystretch}{1} 
\begin{table*}[htb]	
	\centering
	{\begin{tabular}{|c| c c c c c c c|}
		\hline
		\multirow{2}{*}{Method} &\multicolumn{7}{c|}{Number of Samples }\cr\cline{2-8}
		&20&30&50&100&200&300&500\cr\hline
		{ DRSP}  &{1.6}  & {-0.3} & {-1.8}& {-2.5} & {-3.2} & \textbf{-5.2} & \textbf{-5.4}\cr
		{ DRSP-S}&\textbf{-7.1} &\textbf{-6.4} &\textbf{-6.4}	&\textbf{-6.1}	&\textbf{-6.6} &{-5.0} &{-3.3} \cr
		{ M-LB}  &-0.1  &-0.2 &-0.2 &-0.5 &-0.6  &-0.3 &-0.8  \cr
		{ M-UB}  &0.6 &-0.5 &-1.6 &-2.4 &-2.0 &-0.2 &-0.6  \cr
		\hline
	\end{tabular}}
	\caption{Percentage differences (in $\%$) between the DR models and SAA in terms of out-of-sample performance.}
	\label{Com_in_out1}
\end{table*}

\renewcommand{\arraystretch}{1} 
\begin{table*}[htb]	
	\centering
	{\begin{tabular}{|c|c c c c c c c|}
		\hline
		\multirow{2}{*}{Method} &\multicolumn{7}{c|}{Number of Samples}\cr\cline{2-8}
		 &20&30&50&100&200&300&500\cr\hline
		{ DRSP}  &0.13 &0.16 &0.16 &0.23 &1.93 &3.16 &3.85 	\cr
		{ DRSP-S}&0.66 &1.05 &2.02 &5.93 &18.26 &37.48 &100.72  \cr
		{ M-LB}  &68.16 &113.13	&180.66 &413.46 &844.54 &1158.12 &2090.82	 \cr
		{ M-UB}  &3.40 &3.73 &3.71 &6.44 &11.66 &23.20 &29.77 	 \cr
		{ SAA}	 &0.73 &0.83 &1.89 &2.43 &3.98 &7.58 &11.50  \cr
		\hline
	\end{tabular}}
	\caption{Averaged computation time (second) of different methods.}
	\label{Com_in_time1}
\end{table*}

\subsection{Real Road Network Experiments}
{Now we evaluate our DRSP model on a road network with a real dataset. We compare our model with the SAA and the moment-based model. In this experiment, both $\alpha$ and $N$ are the same as these in Section \ref{comparison} and the Wasserstein radius $\epsilon_N$ is tuned according to the sample dataset.}

{We construct a dataset of travel time from Tsinghua University (THU) to Beijing Capital International Airport (BCIA) captured from the AMAP which provides a live traffic data interface\footnote{https://lbs.amap.com/api/webservice/guide/api/direction\#driving}. Firstly we select twenty-one paths which individuals usually take from THU to BCIA as illustrated in Figure \ref{amap}. We set THU as the origin vertex and BCIA as the destination vertex. Moreover, we select several way-points on each path as vertices in the network and then provide an illustrative network of the map in Figure \ref{map_sim}.}

\begin{figure}[htbp]
	\centering
	\includegraphics[scale=0.2]{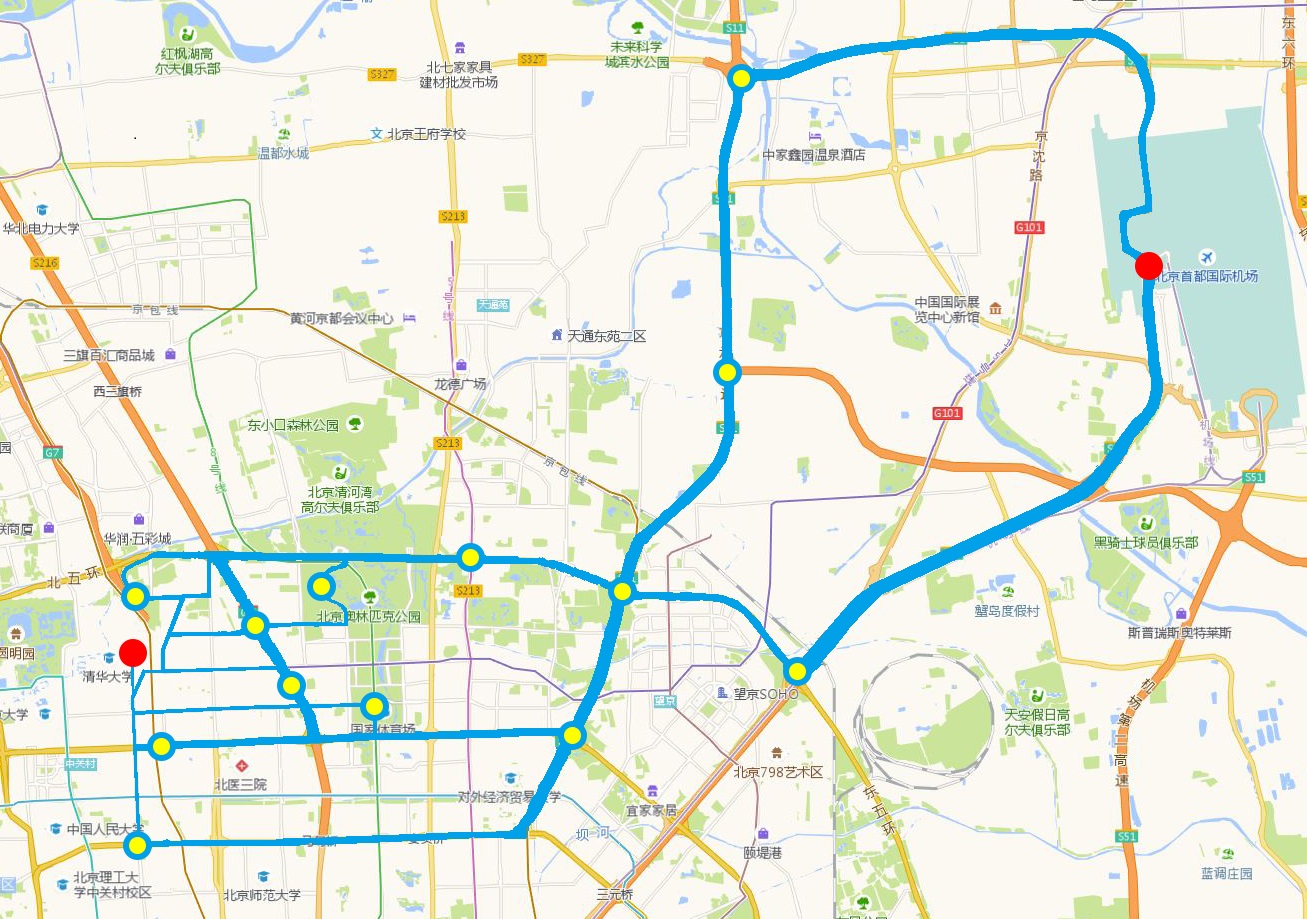}
	\caption{Map with 21 paths from THU to BCIA.}
	\label{amap}
\end{figure}
\begin{figure}[htbp]
	\centering
	\includegraphics[scale=0.4]{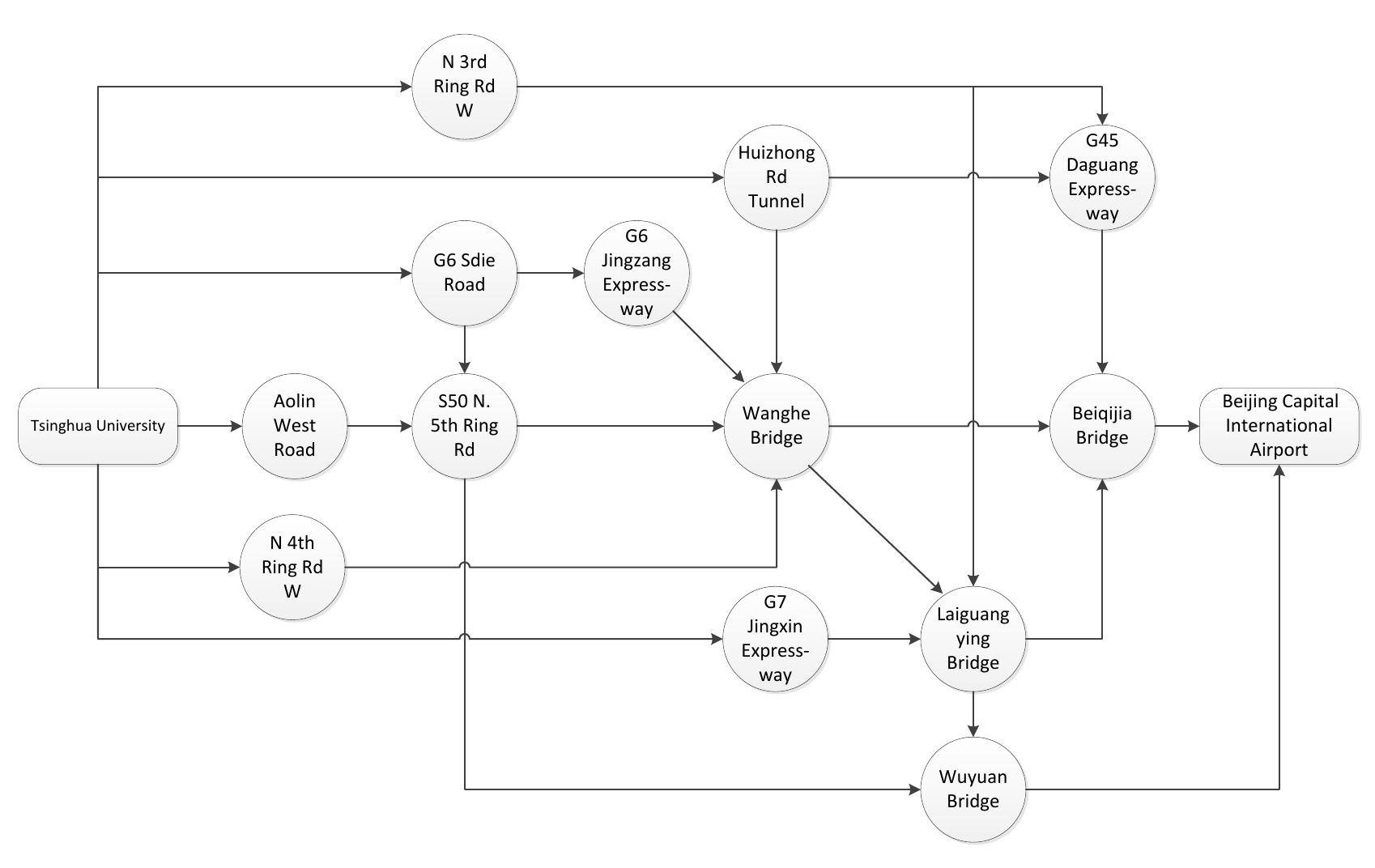}
	\caption{Simplified network between THU and BCIA.}
	\label{map_sim}
\end{figure}

{We collect travel time of each arc along paths as shown in Figure \ref{map_sim} and update it every 10 minutes over a week spanning from Sunday morning, January 06, 2019 to Sunday morning, January 13, 2019. Thus, we have collected $24\cdot6\cdot7=1008$ different observations where each data point contains the travel time of all arcs.}

{Since the true distribution $F$ of the travel time is unknown, it is unable to exactly compute the out-of-sample performance. Similar to Section \ref{DR-SPP}, we utilize $300$ randomly chosen test samples from the dataset to approximate the out-of-sample performance.}

{It should be noticed that we do not consider transportation expenses such as tolls for high-speed and bridge maintenance of roads. Thus, the optimal path may charge more than others. Moreover, the time spent on waiting for the traffic light is ignored as well.}

\renewcommand{\arraystretch}{1} 
\begin{table*}[htb]
	
	\centering
	{\begin{tabular}{|c| c c c c c c c|}
		\hline
		\multirow{2}{*}{Method} &\multicolumn{7}{c|}{Number of Samples }\cr\cline{2-8}
		&20&30&50&100&200&300&500\cr\hline
		{ DRSP} &-0.05 &-0.5 &{-2.3} & {-2.4} & \textbf{-3.5}	&{-3.4} & {-3.2} 	\cr
		{ DRSP-S}&\textbf{-1.3} &\textbf{-2.1} &\textbf{-4.3} &\textbf{-3.8} &{-3.3}	&\textbf{-4.3} &\textbf{-3.7}  \cr
		{ M-LB}  &{0.9} & 0.6 &-1.5 &-2.1 	&-2.3 &-1.9 &-2.0	 \cr
		{ M-UB} & 0.8 & 0.5 &-1.9 &-1.7 &-1.9 &-2.1 &-1.8  \cr
		\hline
	\end{tabular}}
	\caption{Percentage differences (in $\%$) between the DR solutions and SAA solutions for out-of-sample performance.}
	\label{Com_in_out2}
\end{table*}

\renewcommand{\arraystretch}{1} 
\begin{table*}[htb]
	
	\centering
	{\begin{tabular}{|c| c c c c c c c|}
		\hline
		\multirow{2}{*}{Method} &\multicolumn{7}{c|}{Number of Samples }\cr\cline{2-8}
		&20&30&50&100&200&300&500\cr\hline
		{ DRSP}  &0.08 &0.10 &0.11 &0.13 &0.21 &0.33 &0.52	\cr
		{ DRSP-S}&0.11 &0.13 &0.17 &0.27 &0.58 &1.01 &1.82	 \cr
		{ M-LB}  &0.14 &0.31 &0.30 &0.96 &2.98 &6.93 &18.23 \cr
		{ M-UB}  &0.40 &0.33 &0.37 &0.56 &0.84 &1.27 &1.96	 \cr
		{ SAA}	 &0.21 &0.21 &0.25 &0.47 &0.80 &1.20 &2.15	 \cr
		\hline
	\end{tabular}}
	\caption{Averaged computation time (second) of different methods in different experiments.}
	\label{Com_in_time2}
\end{table*}


{Results averaged on $200$ independent simulations are given in Table \ref{Com_in_out2} and Table \ref{Com_in_time2} which confirm the advantages of our DRSP models as expected. }

\section{Conclusion} \label{con}
We have proposed a data-driven DRSP model for finding optimal paths to minimize the w-METT over a Wasserstein ball. {Our DRSP model can be reformulated as a solvable finite convex problem while the DRSP model over the moment-based ambiguity set is an intractable co-positive program. We derived an explicit form of the worst-case distribution in w-METT.  Experimental results validate that the proposed DRSP model provides good out-of-sample performance. Moreover, our model can be extended to the DR bi-criteria shortest path problem and the minimum cost flow problem easily. }

\section*{Acknowledgements}
The authors would like to thank the Associate Editor and anonymous reviewers for their very constructive comments, which greatly improved the quality of this work.
This work is supported by the National Natural Science Foundation of China under grant number 61722308,  U1660202, and 71871023.

\section*{References}

\bibliography{mybib}

\end{document}